\documentclass[12pt,leqno]{amsart}
\usepackage{amsfonts,amsthm,amsmath,comment,tikz-cd,capt-of}
\allowdisplaybreaks

\usepackage{amssymb}

\theoremstyle{plain}
\newtheorem{thm}{Theorem}[section]

\theoremstyle{definition}
\newtheorem{df}{Definition}[section]
\newtheorem{rem}{Remark}[section]
\newtheorem{ex}{Example}[section]

\newcommand{\FF}{\mathbb{F}}
\newcommand{\ZZ}{\mathbb{Z}}

\newcommand{\0}{\mathbf{0}}

\newcommand{\NN}{\mathbb{N}}

\newcommand{\CZ}{\mathcal{Z}}
\newcommand{\G}{\mathcal{G}}
\newcommand{\J}{\mathcal{J}}

\def\bm#1{\mathbf{#1}}

\DeclareMathOperator{\wt}{wt}

\begin{document}

\title{{On the cycle index and the weight enumerator II}
}
\author[Chakraborty]{Himadri Shekhar Chakraborty*}
\thanks{*Corresponding author}
\address
	{
		(1) Graduate School of Natural Science and Technology,
		Kanazawa University,  
		Ishikawa 920-1192, Japan.
		The current affiliation is Department of Mathematics, Shahjalal University of Science and Technology, Sylhet-3114, Bangladesh 
	}
\email{himadri-mat@sust.edu}
\author[Miezaki]{Tsuyoshi Miezaki}
\address{Faculty of Science and Engineering, 
		Waseda University, 
		Tokyo 169--8555, Japan
}
\email{miezaki@waseda.jp} 
\author[Oura]{Manabu Oura}
\address
	{
		Faculty of Mathematics and Physics, 
		Kanazawa University,  
		Ishikawa 920--1192, Japan 
	}
\email{oura@se.kanazawa-u.ac.jp}
\date{\today}

\maketitle

\begin{abstract}
In {a previous paper}, 
the second and third named author introduced the 
concept of the complete cycle index and 
discussed a relation with the 
complete weight enumerator 
in coding theory. 
In the present paper, 
we introduce the 
concept of 
the complete joint cycle index and 
the average complete joint cycle index, and 
discuss a relation with the 
complete joint weight enumerator and 
the average complete joint weight enumerator, respectively 
in coding theory. 
Moreover, 
the notion of the
average intersection numbers is given, 
and we discuss a relation with the 
average intersection numbers in coding theory.
\end{abstract}


\noindent
{\small\bfseries Key Words and Phrases.}
{Codes}, cycle index, {weight enumerator}.\\ \vspace{-0.15in}

\noindent
2010 {\it Mathematics Subject Classification}.
Primary 11T71;
Secondary 20B05, 11H71.\\ \quad

\section{Introduction}\label{Sec:Introduction}

This paper is a sequel to the 
previous paper \cite{MO}. 
In \cite{MO}, 
the first and second named authors defined 
the complete cycle index and 
gave a relationship 
between 
the complete cycle index and 
the complete weight enumerator. 
This was motivated by \cite{CameronPaper,Cameron}, 
which gave 
a relationship between 
the cycle index and the weight enumerator. 

To state our results, 
we review \cite{MO}.
Let $G$ be a permutation group on a set $\Omega$, 
where $|\Omega| = n$. 
For each element
$h \in G$, we can decompose the permutation $h$ into a product 
of disjoint cycles; let
$c(h,i)$ be the number of $i$-cycles occurring by the action of $h$. 
Now the complete cycle
index of $G$ is the polynomial $Z(G;s(h,i): h\in G,i\in \NN)$ in indeterminates 
$\{s(h,i)\mid h\in G,i\in \NN\}$ given by
\[
	Z(G;s(h,i): h\in G,i\in \NN) =
	\sum_{h\in G}
	\prod_{i\in \NN}s(h,i)^{c(h,i)},
\]
where $ \NN := \{x \in \ZZ \mid x > 0\}$.

Let $\FF_q$ be the finite field of order $q$, where $q$ is a prime power. 
An~$\FF_q$-linear code $C$ is a vector subspace of $F_q^n$.
Then the complete weight enumerator of genus $g$ for a $\FF_q$-linear code~$C$ is defined as:
\[
	w_C^{(g)}(x_{{\bf a}}:{\bf a}\in \FF_q^g)=\sum_{{\bf v_1},\ldots,{\bf v_g}\in C}
	\prod_{{\bf a}\in \FF_q^g}x_{{\bf a}}^{n_{{\bf a}}({\bf v_1},\ldots,{\bf v_g})}, 
\]
where $n_{{\bf a}}({\bf v_1},\ldots,{\bf v_g})$ denotes the number 
of $i$ such that ${\bf a}=(v_{1i},\ldots,v_{gi})$.

\begin{df}\label{Def:MiezakiOura}
We construct from $C^g:=\underbrace{C\times \cdots\times C}_{g}$ a permutation
group $G(C^{g})$.
The group we construct is the additive group of $C^{g}$. 
We denote an element of $C^g$ by 
\[
	\tilde{\bm{c}} := 
	({\bm c_1},\ldots,{\bm c_n}):=
	\begin{pmatrix}
		a_{11}&\ldots&a_{1n}\\
		a_{21}&\ldots&a_{2n}\\
		\vdots&\cdots&\vdots\\
		a_{g1}&\ldots&a_{gn}
	\end{pmatrix},
\]
where 
{${\bm c_i}:={}^t(a_{1i},\ldots,a_{gi})\in \FF_q^{g}$}. 
We denote by $\mu_{j}(\tilde{\bm c})$ 
the $j$-th row of $\tilde{{\bm {c}}}$, and 
$\mu_{j}(\tilde{\bm c}) := (a_{j1},\ldots,a_{jn}) \in C$.
We let it act on the set
$\{1,\ldots,n\}\times \FF_q^{g}$ in the following way: 
$({\bf c_1},\ldots,{\bf c_n})$
acts as the permutation
\[
	\left(i,
	\begin{pmatrix}
		x_{1}\\
		x_{2}\\
		\vdots\\
		x_{g}
	\end{pmatrix}
	\right)\mapsto 
	\left(i,
	\begin{pmatrix}
		x_{1}+a_{1i}\\
		x_{2}+a_{2i}\\
		\vdots\\
		x_{g}+a_{gi}
	\end{pmatrix}
	\right) 
\]
of the set $\{1,\ldots,n\}\times \FF_q^g$.  
We call the complete cycle index 
\[
	Z(G(C^{g}),s(h,i):h\in C^g,i\in \NN) 
\]
the complete cycle index of genus $g$ for a code $C$. 
\end{df}

\begin{thm}[{\cite[Theorem 2.1]{MO}}]\label{thm:MiezakiOura}
	Let $C$ be a linear code over $\FF_q$ of length $n$, 
	where $q$ is a power of the prime number $p$.
	Let $w_C^{(g)}(x_{\bf a}:{\bf a}\in \FF_q^g)$ be the 
	complete weight enumerator of genus $g$ and 
	$Z(G(C^g);s(h,i):h\in C^g,i\in\NN)$ be the 
	complete cycle index of genus $g$. 
	
	Let $T$ be a map defined as follows: 
	for each $h=({\bm c_1},\ldots,{\bm c_n})\in C^g$ and $i\in \{1,\ldots,n\}$, 
	if 
	${\bm c_i}={\bm 0}$, then 
	\[
		s(h,1)\mapsto x_{{\bm c_i}}^{1/q^g};
	\]
	if 
	${\bm c_i}\neq {\bm 0}$, then 
	\[
		s(h,p)\mapsto x_{{\bm c_i}}^{p/q^g}. 
	\]
	Then we have 
	\[
		w_C^{(g)}(x_{\bm{a}}:\bm{a}\in \FF_q^g)
		=
		T(Z(G(C^g);s(h,i):h\in C^g,i\in \NN)). 
	\]
\end{thm}

The notion of the joint weight enumerator of two $\FF_q$-linear codes was introduced in~\cite{MMC1972}. Further, the notion of the $g$-fold complete joint weight enumerator of~$g$ linear codes over $\FF_q$ was given in~\cite{SC2000}.

\begin{df}[\cite{SC2000}]\label{Def:CompJWE}
	Let $C$ and $D$ be two linear codes of length $n$ over $\FF_q$. The \emph{complete joint weight enumerator} of codes $C$ and $D$ is defined as follows:
	\[
		\J_{C,D}(x_{\bm a} : \bm a \in \FF_q^2) := 
		\sum_{\bm{u}\in C,\bm{v} \in D}
		\prod_{\bm{a} \in \FF_q^{2}} x_{\bm a}^{n_{\bm a}(\bm{u}, \bm{v})},
	\]
	where $n_{{\bm a}}(\bm{u},\bm{v})$ denotes the number of $i$ such that ${\bf a}=(u_{i},v_{i})$.
\end{df}
The definition above gives rise to a natural question: is there
a complete joint cycle index that relates the 
complete joint weight enumerator $\J_{C,D}(x_{\bm a} : \bm a \in \FF_q^2)$?
The aim of the present paper is to provide a candidate that answers this question. 
We now present the concept of the complete joint cycle index. 

\begin{df}
Let $G$ and $H$ be two permutation groups on a set $\Omega$, where $|\Omega|=n$. Again let $\G_{G,H} := G \times H$ be the direct product of $G$ and $H$. For each element $(g,h) \in \G_{G,H}$, where $g \in G$ and $h \in H$, we can decompose each permutation of the pair $(g,h)$ into a product of disjoint cycles. Let $c(gh,i)$ be the number of $i$-cycles occurring by the action of $gh$, where $gh$ denotes the \emph{product of permutations} $g$ and $h$ which acts on $\Omega$ as $(gh)(\alpha) = h(g(\alpha))$ for any $\alpha \in \Omega$. Now the \emph{complete joint cycle index} of $G$ and $H$ is the polynomial 
\[
	\CZ_{G,H}(s((g,h),i)):=\CZ(\G_{G,H};s((g,h),i): (g,h) \in \G_{G,H}, i\in\NN)
\]
in indeterminates $s((g,h),i)$, where $(g,h) \in \G_{G,H}$ and $i \in \NN$, given by
\[
	\CZ_{G,H}(s((g,h),i)) := \sum_{(g,h) \in \G_{G,H}} \prod_{i \in \NN} s((g,h),i)^{c(gh,i)}.
\]
\end{df}

The concept of the complete joint cycle index is used in Theorem~\ref{Th:CompleteJointCycle}. Theorem~\ref{Th:CompleteJointCycle} gives a relation between complete joint cycle index and complete joint weight enumerator. This generalizes the earlier work Theorem~\ref{thm:MiezakiOura}. Further, we give the notion of the $r$-fold complete joint cycle index and the $(\ell,r)$-fold complete joint weight enumerator. In this paper we also give a link between the $r$-fold complete joint cycle index and the $(\ell,r)$-fold complete joint weight enumerator. The link is the main result of our paper giving a generalization of Theorem~\ref{Th:CompleteJointCycle}. This result presents us a new application of constructing the average $r$-fold complete joint cycle index and a motivation to establish a relation with the average $(\ell,r)$-fold complete joint weight enumerator analogue to the main result.
The diagrams in the following remarks may describe a concrete idea of our paper.
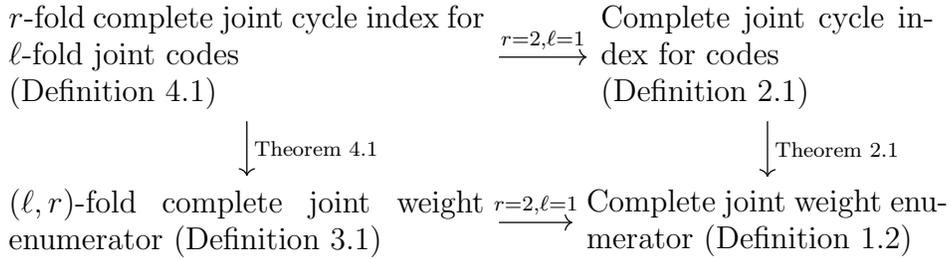
\begin{rem}
	\[
		\begin{tikzcd}
		\parbox{.5\textwidth}{$r$-fold complete joint cycle index for $\ell$-fold joint codes\\ (Definition~\ref{Def:rfoldCJCI})}
		\arrow{r}{r=2,\ell=1}\arrow{d}{\text{Theorem~\ref{Thm:Main}}}  
		& \parbox{.35\textwidth}{Complete joint cycle index for codes\\ (Definition~\ref{Def:CompleteJointCycleIndex})}
		\arrow{d}{\text{Theorem~\ref{Th:CompleteJointCycle}}} \\
		\parbox{.5\textwidth}{$(\ell,r)$-fold complete joint weight enumerator (Definition~\ref{Def:rsfoldCJWE})}
		\arrow{r}{r=2,\ell=1} 
		& \parbox{.38\textwidth}{Complete joint weight enumerator (Definition~\ref{Def:CompJWE})}
		\end{tikzcd}
		\]
\end{rem} 

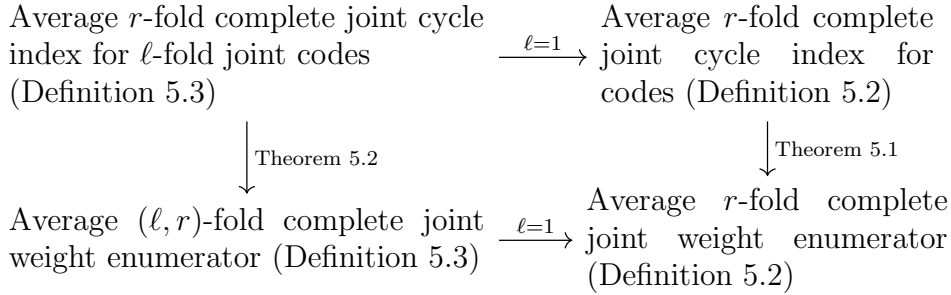
\begin{rem}
	\[
	\begin{tikzcd}
		\parbox{.5\textwidth}{Average $r$-fold complete joint cycle index for $\ell$-fold joint codes \\(Definition~\ref{Def:rsAverageCJWE})}
		\arrow{r}{\ell=1} \arrow{d}{\text{Theorem~\ref{Thm:rsAverage}}}  
		&\parbox{.35\textwidth}{Average $r$-fold complete joint cycle index for codes (Definition~\ref{Def:AverageCJWE})} 
		\arrow{d}{\text{Theorem~\ref{Th:AverageJoint}}} \\
		\parbox{.5\textwidth}{Average $(\ell,r)$-fold complete joint weight enumerator (Definition~\ref{Def:rsAverageCJWE})}
		\arrow{r}{\ell=1} 
		&\parbox{.38\textwidth}{Average $r$-fold complete joint weight enumerator (Definition~\ref{Def:AverageCJWE})}
	\end{tikzcd}
	\]
\end{rem}

This paper is organized as follows. 
In Section~\ref{Sec:TheRelation}, we give a relation between the complete joint cycle index and the complete joint weight enumerator (Theorem~\ref{Th:CompleteJointCycle}). 
In Section~\ref{Sec:rsCompleteJointWeightEnum}, we introduce the notion of the $(\ell,r)$-fold complete joint weight enumerators of $\ell$-fold joint codes. We also give the MacWilliams type identity (Theorem~\ref{Th:GenMacWilliams}) for the $(\ell,r)$-fold complete joint weight enumerators. 
In Section~\ref{Sec:rCompleteJointCycleIndex}, we generalize the notion of the complete joint cycle index to the $r$-fold complete joint cycle index and also give the main result (Theorem~\ref{Thm:Main}) of this paper.
In Section~\ref{Sec:AverageCompleteJointCycleIndex}, we give the concept of the average complete joint cycle index and obtain a relation with the average complete joint weight enumerator of codes (Theorem~\ref{Th:AverageJoint}) and present a generalization (Theorem~\ref{Thm:rsAverage}) of the relation. 
In Section~\ref{Sec:AverageCompleteJointCycleIndex}, we also give the notion of the average intersection number of two permutation groups (Definition~\ref{Def:GrAvIN}), and obtain a relation with the average intersection numbers in coding theory (Theorem~\ref{Th:IntersectionNumber}).
In Section~\ref{Sec:ZkAnalogue}, 
we give a $\ZZ_k$-linear code analogue of the main result (Theorem~\ref{Th:ZkAnalogue}).

\section{The Relation}\label{Sec:TheRelation}

In this section, from any two $\FF_q$-linear codes, we construct two permutation groups, whose complete joint cycle index is essentially the complete joint weight enumerator of codes.
 
\begin{df}\label{Def:CompleteJointCycleIndex}
	Let $C$ and $D$ be two linear codes of length $n$ over $\FF_q$. 
	We construct from $C$ and $D$ two permutation groups $G(C)$ and $H(D)$ respectively. The groups $G(C)$ and $H(D)$ are the additive group of~$C$ and~$D$ respectively. We let each group act on the set $\{1, \dots, n\} \times \FF_q$ in the following way: the codeword $(u_1, \ldots, u_n)$ acts as the permutation 
	\[
		(i,x) \mapsto (i, x+u_i)
	\]
	of the set $\{1, \ldots, n\} \times \FF_q$. We define the \emph{product} of two permutations $(u_1, \ldots, u_n) \in C$ and $(v_1, \ldots, v_n) \in D$ as follows:
	\[
		(i,x) \mapsto (i,x+u_i+v_i)
	\]
	of a set $\{1, \ldots, n\} \times \FF_q$. Let $\G_{C,D} := G(C) \times H(D)$. We call the complete joint cycle index
	\[
		\CZ_{C,D}(s((g,h),i)) := \CZ(\G_{C,D};s((g,h),i): (g,h)\in \G_{C,D}, i\in\NN)
	\]
	the complete joint cycle index for codes $C$ and $D$.
\end{df}

\begin{ex}\label{Ex:Example1}
	Let
	\[
	C:=\left\{
	\begin{pmatrix}
	0\\
	0
	\end{pmatrix}, 
	\begin{pmatrix}
	0\\
	1
	\end{pmatrix}, 
	\begin{pmatrix}
	1\\
	0
	\end{pmatrix}, 
	\begin{pmatrix}
	1\\
	1
	\end{pmatrix} 
	\right\},
	D:=\left\{
	\begin{pmatrix}
	0\\
	0
	\end{pmatrix},
	\begin{pmatrix}
	1\\
	1
	\end{pmatrix} 
	\right\}
	\]
	Then the complete joint weight enumerator is 
	\[
	x_{00}^2+
	x_{01}^2+
	x_{00}x_{10}+
	x_{01}x_{11}+
	x_{10}x_{00}+
	x_{11}x_{01}+
	x_{10}^2+
	x_{11}^2.
	\]
	Let $G(C)$ and $H(D)$ are the permutation groups on $\{1,2\}\times {\FF_2}$. In the following calculation, for $g \in G(C)$ and $h \in H(D)$, we prefer to write the indeterminates $s((g,h),i)$ as 
	\[
		s\left(\dbinom{g}{h},i\right).
	\]
	Then the joint cycle index is 
	\begin{align*}
	&s
	\left(
	\begin{pmatrix}
	0&0\\
	0&0
	\end{pmatrix}, 1
	\right)^2
	s
	\left(
	\begin{pmatrix}
	0&0\\
	0&0
	\end{pmatrix}, 1
	\right)^2 +
	s
	\left(
	\begin{pmatrix}
	0&0\\
	1&1
	\end{pmatrix}, 2
	\right)^1
	s
	\left(
	\begin{pmatrix}
	0&0\\
	1&1
	\end{pmatrix}, 2
	\right)^1\\
	+&
	s
	\left(
	\begin{pmatrix}
	0&1\\
	0&0
	\end{pmatrix}, 1
	\right)^2
	s
	\left(
	\begin{pmatrix}
	0&1\\
	0&0
	\end{pmatrix}, 2
	\right)^1 +
	s
	\left(
	\begin{pmatrix}
	0&1\\
	1&1
	\end{pmatrix}, 2
	\right)^1
	s
	\left(
	\begin{pmatrix}
	0&1\\
	1&1
	\end{pmatrix}, 1
	\right)^2\\
	+&
	s
	\left(
	\begin{pmatrix}
	1&0\\
	0&0
	\end{pmatrix}, 2
	\right)^1
	s
	\left(
	\begin{pmatrix}
	1&0\\
	0&0
	\end{pmatrix}, 1
	\right)^2	+
	s
	\left(
	\begin{pmatrix}
	1&0\\
	1&1
	\end{pmatrix}, 1
	\right)^2
	s
	\left(
	\begin{pmatrix}
	1&0\\
	1&1
	\end{pmatrix}, 2
	\right)^1\\
	+&
	s
	\left(
	\begin{pmatrix}
	1&1\\
	0&0
	\end{pmatrix}, 2
	\right)^1
	s
	\left(
	\begin{pmatrix}
	1&1\\
	0&0
	\end{pmatrix}, 1
	\right)^2	+
	s
	\left(
	\begin{pmatrix}
	1&1\\
	1&1
	\end{pmatrix}, 1
	\right)^2
	s
	\left(
	\begin{pmatrix}
	1&1\\
	1&1
	\end{pmatrix}, 1
	\right)^2.
	\end{align*}
\end{ex}

Now we have the following result.

\begin{thm}\label{Th:CompleteJointCycle}
	Let $C$ and $D$ be two codes over $\FF_q$ of length $n$, where $q$ is a power of the prime number $p$.
	Let $\J_{C,D}(x_{\bm a}:{\bm a}\in \FF_q^2)$ be the complete joint weight enumerator and 
	$\CZ(\G_{C,D};s((g,h),i):(g,h)\in \G_{C,D},i\in\NN)$ be the complete joint cycle index. 
	
	Let $T$ be a map defined as follows: 
	for each $g = (u_1,\ldots,u_n)\in C$ and $h = (v_1,\ldots,v_n)\in D$, and for $i\in \{1,\ldots,n\}$,\\ 
	if 	
	$u_i+v_i = 0$, then 
	\[
		s((g,h),1)\mapsto x_{{u_i v_i}}^{1/q};
	\]
	if 
	$ u_i+v_i\neq 0$, then 
	\[
		s((g,h),p)\mapsto x_{{u_i v_i}}^{p/q}. 
	\]
	Then we have 
	\[
		\J_{C,D}(x_{\bm a}: {\bm a}\in \FF_q^2)
		=
		T(\CZ(\G_{C,D};s((g,h),i):(g,h)\in \G_{C,D},i\in\NN)). 
	\]
\end{thm}

\begin{proof}
	Let $g = (u_1,\ldots,u_n)\in C$ and $h = (v_1,\ldots,v_n)\in D$. Again let
	\[
		\wt(g,h)=\sharp\{i\mid {u_i + v_i}\neq 0\}. 
	\]
	If ${u_i+v_i}= 0$, then the $q$ points of the form 
	$(i,{x})\in \{1,\ldots,n\}\times \FF_q$ 
	are all fixed by these elements; 
	if ${u_i+v_i}\neq 0$, they are permuted in $q/p$ cycles of length $p$, 
	each of the form
	\[
		(i, x) \mapsto (i,x+u_i+v_i ) \mapsto (i,x+2u_i+2v_i ) \mapsto \cdots \mapsto (i,x+ pu_i+pv_i ) = (i,x),
	\]
	the last equation holding because $\FF_q$ is of characteristic $p$. 
	Thus, $g = (u_1,\ldots,u_n)\in C$ and $h = (v_1,\ldots,v_n)\in D$ contribute 
	\[
		s((g,h),1)^{q(n-\wt(g,h))}s((g,h),p)^{(q/p)\wt(g,h)}
	\]
	to the sum in the formula for the complete joint cycle index, 
	and 
	\[
		\prod_{i=1}^{n} x_{{u_i}{v_i}}
	\]
	to the sum in the formula for the complete joint weight enumerator. 
	The result follows.
\end{proof}

\section{$(\ell,r)$-fold Complete Joint Weight Enumerators}\label{Sec:rsCompleteJointWeightEnum}

In this section, 
we give the concept of 
the $(\ell,r)$-fold complete joint weight enumerator over $\FF_q$. We also provide the MacWilliams identity for the $(\ell,r)$-fold complete joint weight enumerators.

\begin{df}\label{Def:rsfoldCJWE}
We denote, $\Pi^{\ell} := C_{1} \times\cdots\times C_{\ell}$,
where $C_{1},\ldots,C_{\ell}$ be the linear codes of length $n$ over $\FF_q$. We call $\Pi^{\ell}$ as \emph{$\ell$-fold joint code} of $C_{1},\ldots,C_{\ell}$.
We denote an element of $\Pi^{\ell}$ by 
\[
	{\tilde{\bm {c}}}:= ({\bm c_{1}},\ldots,{\bm c_{n}}):=
	\begin{pmatrix}
		a_{11}&\ldots&a_{1n}\\
		a_{21}&\ldots&a_{2n}\\
		\vdots&\cdots&\vdots\\
		a_{\ell 1}&\ldots&a_{\ell n}
	\end{pmatrix},
\]
where 
${\bm c_{i}}:={}^t(a_{1i},\ldots,a_{\ell i})\in \FF_q^{\ell}$ 
and 
$\mu_{j}(\tilde{\bm{c}}):= (a_{j1},\ldots, a_{jn}) \in C_{j}$. 

Now let $\Pi_{1}^{\ell},\ldots,\Pi_{r}^{\ell}$ 
be the $\ell$-fold joint codes (not necessarily the same) 
over~$\FF_q$. 
For $k \in \{1,\ldots,r\}$, 
we denote, $\Pi_{k}^{\ell} := C_{k1} \times \cdots \times C_{k\ell}$, 
where $C_{k1},\ldots,C_{k\ell}$ 
be the linear codes of length~$n$ over $\FF_q$.
An element of $\Pi_{k}^{\ell}$ is denoted by
\[
	\tilde{\bm {c}}_{k}:= 
	({\bm c_{k1}},\ldots,{\bm c_{kn}}):=
	\begin{pmatrix}
		a_{11}^{(k)}&\ldots&a_{1n}^{(k)}\\
		a_{21}^{(k)}&\ldots&a_{2n}^{(k)}\\
		\vdots&\cdots&\vdots\\
		a_{\ell 1}^{(k)}&\ldots&a_{\ell n}^{(k)}
	\end{pmatrix},
\]
where 
${\bm c_{ki}}:={}^t(a_{1i}^{(k)},\ldots,a_{\ell i}^{(k)})\in \FF_q^{\ell}$. 
and 
$\mu_{j}(\tilde{\bm{c}}_{k}):= (a_{j1}^{(k)},\ldots, a_{jn}^{(k)}) \in C_{kj}$.
Then the \emph{$(\ell,r)$-fold complete joint weight enumerator} of $\Pi_{1}^{\ell},\ldots,\Pi_{r}^{\ell}$ is defined as follows:
\[
	\J_{\Pi_{1}^{\ell},\ldots,\Pi_{r}^{\ell}}(x_{\bm a} : \bm a \in \FF_q^{\ell \times r}) := 
	\sum_{\tilde{\bm {c}}_1\in \Pi_{1}^{\ell},\ldots,\tilde{\bm {c}}_r \in \Pi_{r}^{\ell}}
	\prod x_{\bm {a}}^{n_{\bm a}(\tilde{\bm {c}}_1,\ldots, \tilde{\bm {c}}_r)},
\]
where ${n_{\bm a}(\tilde{\bm {c}}_1,\ldots, \tilde{\bm {c}}_r)}$ 
denotes the number of $i$ such that 
${\bm a}=(\bm c_{1i},\ldots,\bm c_{ri})$.
For $r=2$ and $\ell = 1$ 
the complete $(\ell,r)$-fold joint weight enumerator coincide with complete joint weight enumerator (Definition~\ref{Def:CompJWE}).
\end{df}



We have the MacWilliams identity for the complete weight enumerator of a code $C$ over $\FF_q$ from~\cite{MMC1972}. Before giving the MacWilliams identity, 
{we would like to review~\cite{MMC1972} for some basic definitions in coding theory}. 
Let $\FF_{q}$ be the finite field,
where $q =  p^{f}$ for some prime number $p$.
Then we define the \emph{inner product} of two vectors $\bm{u},\bm{v} \in \FF_q^n$ as
\[
	\bm{u} \cdot \bm{v} := u_1v_1 + u_2v_2 + \cdots + u_nv_n,
\]
where $\bm{u} = (u_1,u_2,\ldots, u_n)$ and $\bm{v} = (v_1,v_2,\ldots,v_n)$.  
A \emph{character} $\chi$ of $\FF_{q}$ 
is a homomorphism from the additive group $\FF_{q}$ 
to the multiplicative group of non-zero complex numbers.
Now let $F(x)$ be a primitive irreducible polynomial 
of degree $f$ over $\FF_{p}$ 
and let $\lambda$ be a root of $F(x)$.
Then any element $\alpha \in \FF_{q}$ 
has a unique representation as:
\begin{equation}\label{EquAlpha}
	\alpha 
	=
	\alpha_{0} + \alpha_{1} \lambda	
	+ \alpha_{2} \lambda^{2}
	+ \cdots +
	\alpha_{f-1} \lambda^{f-1},
\end{equation} 
where $\alpha_{i} \in \FF_{p}$.
We define $\chi(\alpha) := \eta_{p}^{\alpha_{0}}$, 
where $\eta_{p}$ is the $p$-th primitive root of unity,
and $\alpha_{0}$ is given by Equation~(\ref{EquAlpha}).

\begin{thm}[\cite{MMC1972}]\label{Th:MacWilliams}
	Let $C$ be a linear code of length~$n$ over $\FF_q$. 
	Let $\mathcal{C}_{C}(x_a : a \in \FF_q)$ be the complete weight enumerator of $C$.
	Then we have
	\[
		\mathcal{C}_{C^\perp}
		(x_a : a \in \FF_q) 
		= 
		\dfrac{1}{|C|} 
		T \cdot \mathcal{C}_C(x_a),
	\]
	where $T = \left(\chi(ab)\right)_{a,b \in \FF_q}$.
\end{thm} 

For an $\ell$-fold joint code $\Pi^{\ell}$ over $\FF_q$, 
let ${\Pi^{\ell}}^\perp := C_1^\perp \times \cdots \times C_{\ell}^\perp$ 
and $|{\Pi^{\ell}}| := |C_1|\times \cdots \times |C_{\ell}|$. We again let $\widehat{\Pi}^{\ell}$ be either $\Pi^{\ell}$ or ${\Pi^{\ell}}^\perp$. Then we define
\[
	\delta(\Pi^{\ell},\widehat{\Pi}^{\ell}) := 
	\begin{cases}
		0 & \mbox{if} \quad \widehat{\Pi}^{\ell} = \Pi^{\ell}, \\
		1 & \mbox{if} \quad \widehat{\Pi}^{\ell} = {\Pi^{\ell}}^\perp.  
	\end{cases}
\]

\begin{thm}\label{Th:GenMacWilliams}
	The MacWilliams identity for the $(\ell,r)$-fold complete joint weight enumerator of $\ell$-fold joint codes $\Pi_{1}^{\ell},\ldots,\Pi_{r}^{\ell}$ over $\FF_q$ is given by
	\begin{multline*}
		\J_{\widehat{\Pi}_{1}^{\ell},\ldots,\widehat{\Pi}_{r}^{\ell}}
		({x_{\bm a} : \bm a \in \FF_q^{\ell \times r}}) 
		= 
		\dfrac{1}{|\Pi_{1}^{\ell}|^{\delta(\Pi_{1}^{\ell},\widehat{\Pi}_{1}^{\ell})} \cdots|\Pi_{r}^{\ell}|^{\delta(\Pi_{r}^{\ell},\widehat{\Pi}_{r}^{\ell})}}\\ 
		\left(T^{\delta(\Pi_{1}^{\ell},\widehat{\Pi}_{1}^{\ell})}\right)^{\otimes\ell} 
		\otimes \cdots \otimes 
		\left(T^{\delta(\Pi_{r}^{\ell},\widehat{\Pi}_{r}^{\ell})}\right)^{\otimes\ell}  
		\J_{\Pi_{1}^{\ell},\ldots,\Pi_{r}^{\ell}}(x_{\bm a}).
	\end{multline*}
	{Here $T^{0}$ means the identity matrix~$I$}.
\end{thm}

\begin{proof}
	It is sufficient to show
	\begin{multline*}
		|\Pi_{k}^{\ell}|
		\J_{\Pi_{1}^{\ell},\ldots,{\Pi_{k-1}^{\ell}},{\Pi_{k}^{\ell}}^\perp,{\Pi_{k+1}^{\ell}}\ldots{\Pi_{r}^{\ell}}}
		(x_{\bm a} : \bm a \in \FF_q^{\ell \times r}) = \\ 
		(I^{\otimes\ell} 
		\otimes \cdots \otimes 
		I^{\otimes\ell}
		\otimes
		\underset{k\text{-th}}{T^{\otimes\ell}}
		\otimes
		I^{\otimes\ell}
		\otimes \cdots \otimes 
		I^{\otimes\ell})
		\J_{\Pi_{1}^{\ell},\ldots,\Pi_{k}^{\ell},\ldots,\Pi_{r}^{\ell}}(x_{\bm a}),
	\end{multline*}
	that is, $\widehat{\Pi}_{k}^{\ell} = {\Pi_{k}^{\ell}}^{\perp}$, 
	and for $j \neq k$, $\widehat{\Pi}_{j}^{\ell} = \Pi_{j}^{\ell}$, 
	and $I$ is the identity matrix. 
	Let	
	$\delta_{{\Pi^{\ell}}^\perp}(\bm v) 
	:= 
	\delta_{C_1^\perp}(\mu_1(\bm v)),\ldots,\delta_{C_{\ell}^\perp}(\mu_{\ell}(\bm v))$, 
	where $\bm v \in \FF_q^{\ell\times n}$. Let
	\[
		\delta_{C_j^\perp}(\mu_j(\bm v)) := 
		\begin{cases}
			1 & \mbox{if} \quad \mu_j(\bm v) \in C_j^\perp, \\
			0 & \mbox{otherwise}.
		\end{cases} 
	\]
	Then we have the the following identity
	\[
		\delta_{C_j^\perp}(\mu_j(\bm v)) 
		= \dfrac{1}{|C_j|} \sum_{u_j \in C_j} \chi(u_j \cdot \mu_j(\bm v)).
	\]
	Hence we can write
	\[
		\delta_{{\Pi^{\ell}}^\perp}(\bm v) = 
		\dfrac{1}{|\Pi^{\ell}|} 
		\sum_{\tilde{\bm {c}} \in \Pi^{\ell}} \chi(\langle\tilde{\bm {c}},\bm v\rangle),
	\]
	where $\langle\tilde{\bm {c}},\bm v\rangle := 
	\sum_{i = 1}^{n} \mu_{i}(\tilde{\bm c}) \cdot \mu_i(\bm v)$. 
	Now
	\begin{align*}
		&|\Pi_{k}^{\ell}|
		\J_{\Pi_{1}^{\ell},\ldots,{\Pi_{k-1}^{\ell}},
		{\Pi_{k}^{\ell}}^\perp,{\Pi_{k+1}^{\ell}},\ldots,{\Pi_{r}^{\ell}}}
		(x_{\bm a} : \bm a \in \FF_q^{\ell \times r}) \\ 
		= 
		&|\Pi_{k}^{\ell}|
		\sum_{\tilde{\bm {c}}_1 \in \Pi_{1}^{\ell},\ldots, \tilde{\bm {c}}_{k-1}\in \Pi_{k-1}^{\ell},\tilde{\bm {c}}_{k+1}\in \Pi_{k+1}^{\ell},\ldots,\tilde{\bm {c}}_{r}\in \Pi_{r}^{\ell}} 
		\sum_{\tilde{\bm{d}}_{k} \in {\Pi_{k}^{\ell}}^\perp} 
		\prod_{\bm a \in \FF_q^{\ell \times r}} 
		x_{\bm a}^{n_{\bm a}(\tilde{\bm{c}}_1,\ldots,\tilde{\bm{d}}_k,\ldots,\tilde{\bm{c}}_r)} \\
		= 
		&|\Pi_{k}^{\ell}|
		\sum_{\tilde{\bm {c}}_1,\ldots, \tilde{\bm {c}}_{k-1},\tilde{\bm {c}}_{k+1},\ldots, \tilde{\bm {c}}_{r}}
		\sum_{\bm v \in \FF_q^{\ell \times n}} 
		\delta_{{\Pi_{k}^{\ell}}^\perp}(\bm v) 
		\prod_{\bm a \in \FF_q^{\ell \times r}} 
		x_{\bm a}^{n_{\bm a}(\tilde{\bm{c}}_1,\ldots,\tilde{\bm{c}}_{k-1}, \bm{v},\tilde{\bm{c}}_{k+1},\ldots,\tilde{\bm{c}}_r)} \\
		= 
		& \sum_{\tilde{\bm {c}}_1, \ldots, \tilde{\bm {c}}_{k-1},\tilde{\bm {c}}_{k+1},\ldots,\tilde{\bm {c}}_{r}}
		\sum_{\bm v \in \FF_q^{\ell \times n}}
		\sum_{\tilde{\bm {c}}_k \in \Pi_{k}^{\ell}} 
		\chi(\langle\tilde{\bm {c}}_k,\bm v\rangle)
		\prod_{\bm a \in \FF_q^{\ell \times r}} 
		x_{\bm a}^
		{n_{\bm a}(\tilde{\bm{c}}_1,\ldots,\tilde{\bm{c}}_{k-1}, \bm{v},\tilde{\bm{c}}_{k+1},\ldots,\tilde{\bm{c}}_r)} \\
		= 
		& \sum_{\tilde{\bm {c}}_1,\ldots, \tilde{\bm {c}}_{k-1},\tilde{\bm{c}}_{k},\tilde{\bm {c}}_{k+1},\ldots,\tilde{\bm {c}}_{r}}
		\sum_{\bm v \in \FF_q^{\ell \times n}}
		\chi(\langle\tilde{\bm {c}}_k, \bm v\rangle)
		\prod_{\bm a \in \FF_q^{\ell \times r}} 
		x_{\bm a}^{n_{\bm a}(\tilde{\bm{c}}_1,\ldots,\tilde{\bm{c}}_{k-1}, \bm{v},\tilde{\bm{c}}_{k+1},\ldots,\tilde{\bm{c}}_r)} \\
		= 
		& \sum_{\tilde{\bm {c}}_1,\ldots, \tilde{\bm {c}}_{k-1}, \tilde{\bm{c}}_{k},			\tilde{\bm{c}}_{k+1},\ldots,\tilde{\bm {c}}_{r}}
		\sum_{\bm v \in \FF_q^{\ell \times n}}
		\chi(\mu_{1}(\tilde{\bm{c}}_k) \cdot \mu_{1}(\bm{v}) + \dots + \mu_{\ell}(\tilde{\bm{c}}_k) \cdot \mu_{\ell}(\bm{v}))\\ 
		&\hspace{4 cm}
		\prod_{\bm a \in \FF_q^{\ell \times r}} 
		x_{\bm a}^{n_{\bm a}(\tilde{\bm{c}}_1,\ldots,\tilde{\bm{c}}_{k-1}, \bm{v},\tilde{\bm{c}}_{k+1},\ldots,\tilde{\bm{c}}_r)} \\
		= 
		&\sum_{\tilde{\bm {c}}_1,\ldots, \tilde{\bm {c}}_{k-1},\tilde{\bm{c}}_{k},\tilde{\bm {c}}_{k+1},\ldots,\tilde{\bm {c}}_{r}}
		\sum_{\bm v_1, \ldots, \bm v_n \in \FF_q^{\ell}} 
		\chi(\bm c_{k1} \cdot \bm v_1 + \cdots + \bm c_{kn} \cdot \bm v_n)\\
		&\hspace{4 cm} 
		\prod_{1\leq i \leq n} 
		x_{\bm c_{1i}\ldots \bm c_{(k-1)i} \bm v_{i} \bm c_{(k+1)i}\ldots \bm c_{ri}} \\
		= 
		&\sum_{\tilde{\bm {c}}_1,\ldots, \tilde{\bm {c}}_{k-1},\tilde{\bm{c}}_{k},\tilde{\bm {c}}_{k+1},\ldots,\tilde{\bm {c}}_{r}}
		\prod_{1\leq i \leq n} 
		\left(\sum_{\bm v_i \in \FF_q^{\ell}} 
		\chi(\bm c_{ki} \cdot \bm v_i) 
		x_{\bm c_{1i}\ldots \bm c_{(k-1)i} \bm v_{i} \bm c_{(k+1)i}\ldots \bm c_{ri}}\right) \\
		= 
		&\sum_{\tilde{\bm {c}}_1,\ldots, \tilde{\bm {c}}_{k-1},\tilde{\bm{c}}_{k},\tilde{\bm {c}}_{k+1},\ldots,\tilde{\bm {c}}_{r}}
		\prod_{\bm a = (\bm a_{1},\ldots,\bm a_{k},\ldots,\bm a_{r}) \in \FF_q^{\ell \times r}}\\
		& \hspace{3 cm} 
		\left(\sum_{\bm v \in \FF_q^{\ell}} 
		\chi(\bm a_{k} \cdot \bm v) 
		x_{\bm a_1 \ldots \bm a_{k-1} \bm v \bm a_{k+1} \ldots \bm a_{r}}\right)^
		{n_{\bm a}(\tilde{\bm {c}}_1,\ldots,\tilde{\bm{c}}_k,\ldots,\tilde{\bm{c}}_r)}\\
		= 
		&\J_{\Pi_{1}^{\ell},\ldots,\Pi_{k}^{\ell},\ldots,\Pi_{r}^{\ell}}\\
		&\hspace{0.2 cm}
		\left(\sum_{\bm v \in \FF_q^{\ell}} \chi(\bm a_{k} \cdot \bm v) 
		x_{\bm a_1 \ldots \bm a_{k-1} \bm v \bm a_{k+1} \ldots \bm a_{r}}
		: 
		(\bm a_{1},\ldots,\bm a_{k},\ldots,\bm a_{r}) \in \FF_q^{\ell \times r}\right) \\
		=
		&(I^{\otimes\ell}
		\otimes \cdots \otimes 
		I^{\otimes\ell}
		\otimes
		T^{\otimes\ell}
		\otimes
		I^{\otimes\ell}
		\otimes \cdots \otimes 
		I^{\otimes\ell})
		\J_{\Pi_{1}^{\ell},\ldots,\Pi_{k}^{\ell},\ldots,\Pi_{r}^{\ell}}(x_{\bm a})
	\end{align*}
	as it should. Hence the proof is completed.
\end{proof}

\section{$r$-fold Complete Joint Cycle Index}\label{Sec:rCompleteJointCycleIndex}

Let $G_1,G_2,\ldots,G_r$ be $r$ permutation groups on a set $\Omega$, where $|\Omega|=n$. Again let $\G_{G_1,\ldots,G_r} := G_1 \times \cdots \times G_r$ be the direct product of $G_1,G_2,\ldots,G_r$. For any element $(g_1,g_2,\ldots,g_r) \in \G_{G_1,\ldots,G_r}$, where $g_k \in G_k$ for $k \in \{1,2, \ldots, r\}$, we can decompose each permutation $g_k$ into a product of disjoint cycles. Let $c(g_k,i)$ be the number of $i$-cycles occurring by the action of~$g_k$. Now the \emph{$r$-fold complete joint cycle index} of $G_1,G_2,\ldots,G_r$ is the polynomial 
\begin{align*}
	&\CZ_{G_1,\ldots,G_r}(s((g_{1},\ldots,g_{r}),i))\\ &:=\CZ(\G_{G_1,\ldots,G_r};s((g_{1},\ldots,g_{r}),i): (g_{1},\ldots,g_{r}) \in \G_{G_1,\ldots,G_r}, i\in\NN)
\end{align*}
in indeterminates $s((g_{1},\ldots,g_{r}),i)$, where $(g_{1},\ldots,g_{r}) \in \G_{G_1,\ldots,G_r}$ and $i \in \NN$, given by
\begin{align*}
	& \CZ_{G_1,\ldots,G_r}(s((g_{1},\ldots,g_{r}),i))\\
	& := \sum_{(g_1,\ldots,g_r) \in \G_{G_1,\ldots,G_r}} \prod_{i \in \NN} s((g_1,\ldots,g_r),i)^{c(g_1 \cdots g_r,i)}.
\end{align*}
where $g_1\cdots g_r $ denotes the \emph{product of permutations} 
$g_1,\ldots,g_r$ which acts on $\Omega$ as 
$(g_1\cdots g_r)(\alpha) = g_r(\cdots g_{1}(\alpha)\cdots)$ 
for any $\alpha \in \Omega$. If~$G_1=\cdots=G_r=G$ (say), then we call $\CZ(\G_{G,\ldots,G};s((g_{1},\ldots,g_{r}),i): (g_{1},\ldots,g_{r}) \in \G_{G,\ldots,G}, i\in\NN)$ 
the \emph{$r$-fold complete multi-joint cycle index} of~$G$.

\begin{df}\label{Def:rfoldCJCI}
We construct from $\Pi^{\ell}$ a permutation group $G(\Pi^{\ell})$.
The group we construct is the additive group of $\Pi^{\ell}$.
We let it act on the set 
$\{1,\ldots,n\}\times \FF_q^{\ell}$ in the following way:
$({\bm c_{1}},\ldots,{\bm c_{n}})$
acts as the permutation
\[
	\left(i,
	\begin{pmatrix}
		x_{1}\\
		x_{2}\\
		\vdots\\
		x_{\ell}
	\end{pmatrix}
	\right)\mapsto 
	\left(i,
	\begin{pmatrix}
		x_{1}+a_{1i}\\
		x_{2}+a_{2i}\\
		\vdots\\
		x_{\ell}+a_{\ell i}
	\end{pmatrix}
	\right) 
\]
of the set $\{1,\ldots,n\}\times \FF_q^{\ell}$.
Now let $G_{1}(\Pi_{1}^{\ell}),\ldots,G_{r}(\Pi_{r}^{\ell})$ be $r$ permutation groups.
We define the \emph{product} of $r$ permutations $({\bm c_{11}}, \ldots, {\bm c_{1n}}) \in \Pi_{1}^{\ell}, \ldots, ({\bm c_{r1}}, \ldots, {\bm c_{rn}}) \in \Pi_{r}^{\ell}$ as follows:
\[
	\left(i,
	\begin{pmatrix}
		x_{1}\\
		x_{2}\\
		\vdots\\
		x_{\ell}
	\end{pmatrix}
	\right)\mapsto 
	\left(i,
	\begin{pmatrix}
		x_{1}+\sum_{k=1}^{r}a_{1i}^{(k)}\\
		x_{2}+\sum_{k=1}^{r}a_{2i}^{(k)}\\
		\vdots\\
		x_{\ell}+\sum_{k=1}^{r}a_{\ell i}^{(k)}
	\end{pmatrix}
	\right) 
\]
of a set $\{1, \dots, n\} \times \FF_q^{\ell}$. 
Let $\G_{\Pi_{1}^{\ell},\ldots,\Pi_{r}^{\ell}} := G_{1}(\Pi_{1}^{\ell}) \times \cdots \times G_{r}(\Pi_{r}^{\ell})$. We call the $r$-fold complete joint cycle index 
\begin{multline*}
	\CZ_{\Pi_{1}^{\ell},\ldots,\Pi_{r}^{\ell}}(s((g_1,\ldots,g_r),i))\\ :=\CZ(\G_{\Pi_{1}^{\ell},\ldots,\Pi_{r}^{\ell}};s((g_1,\ldots,g_r),i): (g_1,\ldots,g_r)\in \G_{\Pi_{1}^{\ell},\ldots,\Pi_{r}^{\ell}}, i\in\NN)
\end{multline*}
the $r$-fold complete joint cycle index for $\Pi_{1}^{\ell},\ldots,\Pi_{r}^{\ell}$.
\end{df}

\begin{rem}
	Let $\Pi_{1}^{\ell} = \cdots = \Pi_{r}^{\ell} = \Pi^{\ell}$, where
	\[
		\Pi^{\ell} := C_1\times\cdots \times C_{\ell},
	\]
	for the $\FF_q$-linear codes $C_1,\ldots,C_{\ell}$ of length~$n$. Then we call
	\[
		\CZ(\G_{\Pi^{\ell},\ldots,\Pi^{\ell}};s((g_1,\ldots,g_r),i): (g_1,\ldots,g_r)\in \G_{\Pi^{\ell},\ldots,\Pi^{\ell}}, i\in\NN)
	\]
	the \emph{$r$-fold complete multi-joint cycle index for $\Pi^{\ell}$}. 
	
	Again let $C_1 = \cdots = C_{\ell} = C$, for some $\FF_q$-linear code $C$ of length $n$. Then we denote $\Pi^{\ell}$ by $C^{\ell}$, that is,	
	\[
		C^{\ell}:=\underbrace{C\times \cdots\times C}_{\ell}.
	\]
	We call 
	$\CZ(\G_{C^{\ell},\ldots,C^{\ell}};s((g_1,\ldots,g_r),i): (g_1,\ldots,g_r)\in \G_{C^{\ell},\ldots,C^{\ell}}, i\in\NN)$
	the \emph{$r$-fold complete multi-joint cycle index for $C^{\ell}$}. Note that if $r = 1$,  the 
	$r$-fold complete multi-joint cycle index for $C^{\ell}$ 
	coincide with 
	the \emph{complete cycle index of genus $\ell$} for code~$C$ (Definition~\ref{Def:MiezakiOura}).
\end{rem}

We now give an example.

\begin{ex}\label{Ex:Example2}
	Let
	\begin{align*}
		C:=\left\{
		\begin{pmatrix}
			0\\
			0
		\end{pmatrix}, 
		\begin{pmatrix}
			0\\
			1
		\end{pmatrix}, 
		\begin{pmatrix}
			1\\
			0
		\end{pmatrix}, 
		\begin{pmatrix}
			1\\
			1
		\end{pmatrix} 
		\right\},
		D:=\left\{
		\begin{pmatrix}
			0\\
			0
		\end{pmatrix},
		\begin{pmatrix}
			1\\
			1
		\end{pmatrix} 
		\right\}.
	\end{align*}
	Now let 
	\begin{align*}
		\Pi_{1}^{2} & := C \times D \\
		& = 
		\left\{
			\begin{pmatrix}
				0 & 0\\
				0 & 0
			\end{pmatrix}, 
			\begin{pmatrix}
				0 & 0\\
				1 & 1
			\end{pmatrix}, 
			\begin{pmatrix}
				0 & 1\\
				0 & 0
			\end{pmatrix}, 
			\begin{pmatrix}
				0 & 1\\
				1 & 1
			\end{pmatrix},
			\begin{pmatrix}
				1 & 0\\
				0 & 0
			\end{pmatrix},
			\begin{pmatrix}
				1 & 0\\
				1 & 1
			\end{pmatrix},\right. \\
			& \quad \quad
			\left.\begin{pmatrix}
				1 & 1\\
				0 & 0
			\end{pmatrix},
			\begin{pmatrix}
				1 & 1\\
				1 & 1
			\end{pmatrix}
		\right\}, 
		\text{ and }\\
		\Pi_{2}^{2} & := D \times D\\
		& =
		\left\{
			\begin{pmatrix}
				0 & 0\\
				0 & 0
			\end{pmatrix}, 
			\begin{pmatrix}
				0 & 0\\
				1 & 1
			\end{pmatrix}, 
			\begin{pmatrix}
				1 & 1\\
				0 & 0
			\end{pmatrix},
			\begin{pmatrix}
				1 & 1\\
				1 & 1
			\end{pmatrix}
		\right\}.
	\end{align*}
	Then the $(2,2)$-fold complete joint weight enumerator of $\Pi_{1}^{2}$ and $\Pi_{2}^{2}$ is 
	\begin{align*}
		& x_{\substack{00\\00}}^2 + x_{\substack{00\\01}}^2 + x_{\substack{01\\00}}^2 + x_{\substack{01\\01}}^2 +
		x_{\substack{00\\10}}^2 + x_{\substack{00\\11}}^2 + x_{\substack{01\\10}}^2 + x_{\substack{01\\11}}^2 +\\
		& x_{\substack{00\\00}}x_{\substack{10\\00}} + x_{\substack{00\\01}}x_{\substack{10\\01}} + x_{\substack{01\\00}}x_{\substack{11\\00}} + x_{\substack{01\\01}}x_{\substack{11\\01}} +
		x_{\substack{00\\10}}x_{\substack{10\\10}} + x_{\substack{00\\11}}x_{\substack{10\\11}} + x_{\substack{01\\10}}x_{\substack{11\\10}} + x_{\substack{01\\11}}x_{\substack{11\\11}} + \\
		& x_{\substack{10\\00}}x_{\substack{00\\00}} + x_{\substack{10\\01}}x_{\substack{00\\01}} + x_{\substack{11\\00}}x_{\substack{01\\00}} + x_{\substack{11\\01}}x_{\substack{01\\01}} +
		x_{\substack{10\\10}}x_{\substack{00\\10}} + x_{\substack{10\\11}}x_{\substack{00\\11}} + x_{\substack{11\\10}}x_{\substack{01\\10}} + x_{\substack{11\\11}}x_{\substack{01\\11}} + \\
		& x_{\substack{10\\00}}^2 + x_{\substack{10\\01}}^2 + x_{\substack{11\\00}}^2 + x_{\substack{11\\01}}^2 +
		x_{\substack{10\\10}}^2 + x_{\substack{10\\11}}^2 + x_{\substack{11\\10}}^2 + x_{\substack{11\\11}}^2.
	\end{align*}
	Let $G(\Pi_{1}^{2})$ and $H(\Pi_{2}^{2})$ are the permutation groups on $\{1,2\}\times {\FF_2^2}$.
	Then  
	\begin{align*}
		\CZ_{\Pi_{1}^{2},\Pi_{2}^{2}}&(s((g,h),i)) \\
		& = 
		s
		\left(\left(
		\begin{pmatrix}
			0 & 0\\
			0 & 0
		\end{pmatrix},
		\begin{pmatrix}
			0 & 0\\
			0 & 0
		\end{pmatrix}\right), 1
		\right)^4
		s
		\left(\left(
		\begin{pmatrix}
			0 & 0\\
			0 & 0
		\end{pmatrix},
		\begin{pmatrix}
			0 & 0\\
			0 & 0
		\end{pmatrix}\right), 1
		\right)^4\\
		&+
		s
		\left(\left(
		\begin{pmatrix}
			0 & 0\\
			0 & 0
		\end{pmatrix},
		\begin{pmatrix}
			0 & 0\\
			1 & 1
		\end{pmatrix}\right), 2
		\right)^2
		s
		\left(\left(
		\begin{pmatrix}
			0 & 0\\
			0 & 0
		\end{pmatrix},
		\begin{pmatrix}
			0 & 0\\
			1 & 1
		\end{pmatrix}\right), 2
		\right)^2\\
		&+
		s
		\left(\left(
		\begin{pmatrix}
			0 & 0\\
			0 & 0
		\end{pmatrix},
		\begin{pmatrix}
			1 & 1\\
			0 & 0
		\end{pmatrix}\right), 2
		\right)^2
		s
		\left(\left(
		\begin{pmatrix}
			0 & 0\\
			0 & 0
		\end{pmatrix},
		\begin{pmatrix}
			1 & 1\\
			0 & 0
		\end{pmatrix}\right), 2
		\right)^2\\
		&+
		s
		\left(\left(
		\begin{pmatrix}
			0 & 0\\
			0 & 0
		\end{pmatrix},
		\begin{pmatrix}
			1 & 1\\
			1 & 1
		\end{pmatrix}\right), 2
		\right)^2
		s
		\left(\left(
		\begin{pmatrix}
			0 & 0\\
			0 & 0
		\end{pmatrix},
		\begin{pmatrix}
			1 & 1\\
			1 & 1
		\end{pmatrix}\right), 2
		\right)^2\\
		&+
		s
		\left(\left(
		\begin{pmatrix}
			0 & 0\\
			1 & 1
		\end{pmatrix},
		\begin{pmatrix}
			0 & 0\\
			0 & 0
		\end{pmatrix}\right), 2
		\right)^2
		s
		\left(\left(
		\begin{pmatrix}
		0 & 0\\
		1 & 1
		\end{pmatrix},
		\begin{pmatrix}
		0 & 0\\
		0 & 0
		\end{pmatrix}\right), 2
		\right)^2\\
		&+
		s
		\left(\left(
		\begin{pmatrix}
		0 & 0\\
		1 & 1
		\end{pmatrix},
		\begin{pmatrix}
		0 & 0\\
		1 & 1
		\end{pmatrix}\right), 1
		\right)^4
		s
		\left(\left(
		\begin{pmatrix}
		0 & 0\\
		1 & 1
		\end{pmatrix},
		\begin{pmatrix}
		0 & 0\\
		1 & 1
		\end{pmatrix}\right), 1
		\right)^4\\
		&+
		s
		\left(\left(
		\begin{pmatrix}
		0 & 0\\
		1 & 1
		\end{pmatrix},
		\begin{pmatrix}
		1 & 1\\
		0 & 0
		\end{pmatrix}\right), 2
		\right)^2
		s
		\left(\left(
		\begin{pmatrix}
		0 & 0\\
		1 & 1
		\end{pmatrix},
		\begin{pmatrix}
		1 & 1\\
		0 & 0
		\end{pmatrix}\right), 2
		\right)^2\\
		&+
		s
		\left(\left(
		\begin{pmatrix}
		0 & 0\\
		1 & 1
		\end{pmatrix},
		\begin{pmatrix}
		1 & 1\\
		1 & 1
		\end{pmatrix}\right), 2
		\right)^2
		s
		\left(\left(
		\begin{pmatrix}
		0 & 0\\
		1 & 1
		\end{pmatrix},
		\begin{pmatrix}
		1 & 1\\
		1 & 1
		\end{pmatrix}\right), 2
		\right)^2\\
		&+
		s
		\left(\left(
		\begin{pmatrix}
		0 & 1\\
		0 & 0
		\end{pmatrix},
		\begin{pmatrix}
		0 & 0\\
		0 & 0
		\end{pmatrix}\right), 1
		\right)^4
		s
		\left(\left(
		\begin{pmatrix}
		0 & 1\\
		0 & 0
		\end{pmatrix},
		\begin{pmatrix}
		0 & 0\\
		0 & 0
		\end{pmatrix}\right), 2
		\right)^2\\
		&+
		s
		\left(\left(
		\begin{pmatrix}
		0 & 1\\
		0 & 0
		\end{pmatrix},
		\begin{pmatrix}
		0 & 0\\
		1 & 1
		\end{pmatrix}\right), 2
		\right)^2
		s
		\left(\left(
		\begin{pmatrix}
		0 & 1\\
		0 & 0
		\end{pmatrix},
		\begin{pmatrix}
		0 & 0\\
		1 & 1
		\end{pmatrix}\right), 2
		\right)^2\\
		&+
		s
		\left(\left(
		\begin{pmatrix}
		0 & 1\\
		0 & 0
		\end{pmatrix},
		\begin{pmatrix}
		1 & 1\\
		0 & 0
		\end{pmatrix}\right), 2
		\right)^2
		s
		\left(\left(
		\begin{pmatrix}
		0 & 1\\
		0 & 0
		\end{pmatrix},
		\begin{pmatrix}
		1 & 1\\
		0 & 0
		\end{pmatrix}\right), 1
		\right)^4\\
		&+
		s
		\left(\left(
		\begin{pmatrix}
		0 & 1\\
		0 & 0
		\end{pmatrix},
		\begin{pmatrix}
		1 & 1\\
		1 & 1
		\end{pmatrix}\right), 2
		\right)^2
		s
		\left(\left(
		\begin{pmatrix}
		0 & 1\\
		0 & 0
		\end{pmatrix},
		\begin{pmatrix}
		1 & 1\\
		1 & 1
		\end{pmatrix}\right), 2
		\right)^2\\
		&+
		s
		\left(\left(
		\begin{pmatrix}
		0 & 1\\
		1 & 1
		\end{pmatrix},
		\begin{pmatrix}
		0 & 0\\
		0 & 0
		\end{pmatrix}\right), 2
		\right)^2
		s
		\left(\left(
		\begin{pmatrix}
		0 & 1\\
		1 & 1
		\end{pmatrix},
		\begin{pmatrix}
		0 & 0\\
		0 & 0
		\end{pmatrix}\right), 2
		\right)^2\\
		&+
		s
		\left(\left(
		\begin{pmatrix}
		0 & 1\\
		1 & 1
		\end{pmatrix},
		\begin{pmatrix}
		0 & 0\\
		1 & 1
		\end{pmatrix}\right), 1
		\right)^4
		s
		\left(\left(
		\begin{pmatrix}
		0 & 1\\
		1 & 1
		\end{pmatrix},
		\begin{pmatrix}
		0 & 0\\
		1 & 1
		\end{pmatrix}\right), 2
		\right)^2\\
		&+
		s
		\left(\left(
		\begin{pmatrix}
		0 & 1\\
		1 & 1
		\end{pmatrix},
		\begin{pmatrix}
		1 & 1\\
		0 & 0
		\end{pmatrix}\right), 2
		\right)^2
		s
		\left(\left(
		\begin{pmatrix}
		0 & 1\\
		1 & 1
		\end{pmatrix},
		\begin{pmatrix}
		1 & 1\\
		0 & 0
		\end{pmatrix}\right), 2
		\right)^2\\
		&+
		s
		\left(\left(
		\begin{pmatrix}
		0 & 1\\
		1 & 1
		\end{pmatrix},
		\begin{pmatrix}
		1 & 1\\
		1 & 1
		\end{pmatrix}\right), 2
		\right)^2
		s
		\left(\left(
		\begin{pmatrix}
		0 & 1\\
		1 & 1
		\end{pmatrix},
		\begin{pmatrix}
		1 & 1\\
		1 & 1
		\end{pmatrix}\right), 1
		\right)^4\\
		&+
		s
		\left(\left(
		\begin{pmatrix}
		1 & 0\\
		0 & 0
		\end{pmatrix},
		\begin{pmatrix}
		0 & 0\\
		0 & 0
		\end{pmatrix}\right), 2
		\right)^2
		s
		\left(\left(
		\begin{pmatrix}
		1 & 0\\
		0 & 0
		\end{pmatrix},
		\begin{pmatrix}
		0 & 0\\
		0 & 0
		\end{pmatrix}\right), 1
		\right)^4\\
		&+
		s
		\left(\left(
		\begin{pmatrix}
		1 & 0\\
		0 & 0
		\end{pmatrix},
		\begin{pmatrix}
		0 & 0\\
		1 & 1
		\end{pmatrix}\right), 2
		\right)^2
		s
		\left(\left(
		\begin{pmatrix}
		1 & 0\\
		0 & 0
		\end{pmatrix},
		\begin{pmatrix}
		0 & 0\\
		1 & 1
		\end{pmatrix}\right), 2
		\right)^2\\
		&+
		s
		\left(\left(
		\begin{pmatrix}
		1 & 0\\
		0 & 0
		\end{pmatrix},
		\begin{pmatrix}
		1 & 1\\
		0 & 0
		\end{pmatrix}\right), 1
		\right)^4
		s
		\left(\left(
		\begin{pmatrix}
		1 & 0\\
		0 & 0
		\end{pmatrix},
		\begin{pmatrix}
		1 & 1\\
		0 & 0
		\end{pmatrix}\right), 2
		\right)^2\\
		&+
		s
		\left(\left(
		\begin{pmatrix}
		1 & 0\\
		0 & 0
		\end{pmatrix},
		\begin{pmatrix}
		1 & 1\\
		1 & 1
		\end{pmatrix}\right), 2
		\right)^2
		s
		\left(\left(
		\begin{pmatrix}
		1 & 0\\
		0 & 0
		\end{pmatrix},
		\begin{pmatrix}
		1 & 1\\
		1 & 1
		\end{pmatrix}\right), 2
		\right)^2\\
		&+
		s
		\left(\left(
		\begin{pmatrix}
		1 & 0\\
		1 & 1
		\end{pmatrix},
		\begin{pmatrix}
		0 & 0\\
		0 & 0
		\end{pmatrix}\right), 2
		\right)^2
		s
		\left(\left(
		\begin{pmatrix}
		1 & 0\\
		1 & 1
		\end{pmatrix},
		\begin{pmatrix}
		0 & 0\\
		0 & 0
		\end{pmatrix}\right), 2
		\right)^2\\
		&+
		s
		\left(\left(
		\begin{pmatrix}
		1 & 0\\
		1 & 1
		\end{pmatrix},
		\begin{pmatrix}
		0 & 0\\
		1 & 1
		\end{pmatrix}\right), 2
		\right)^2
		s
		\left(\left(
		\begin{pmatrix}
		1 & 0\\
		1 & 1
		\end{pmatrix},
		\begin{pmatrix}
		0 & 0\\
		1 & 1
		\end{pmatrix}\right), 1
		\right)^4\\
		&+
		s
		\left(\left(
		\begin{pmatrix}
		1 & 0\\
		1 & 1
		\end{pmatrix},
		\begin{pmatrix}
		1 & 1\\
		0 & 0
		\end{pmatrix}\right), 2
		\right)^2
		s
		\left(\left(
		\begin{pmatrix}
		1 & 0\\
		1 & 1
		\end{pmatrix},
		\begin{pmatrix}
		1 & 1\\
		0 & 0
		\end{pmatrix}\right), 2
		\right)^2\\
		&+
		s
		\left(\left(
		\begin{pmatrix}
		1 & 0\\
		1 & 1
		\end{pmatrix},
		\begin{pmatrix}
		1 & 1\\
		1 & 1
		\end{pmatrix}\right), 1
		\right)^4
		s
		\left(\left(
		\begin{pmatrix}
		1 & 0\\
		1 & 1
		\end{pmatrix},
		\begin{pmatrix}
		1 & 1\\
		1 & 1
		\end{pmatrix}\right), 2
		\right)^2\\
		&+
		s
		\left(\left(
		\begin{pmatrix}
		1 & 1\\
		0 & 0
		\end{pmatrix},
		\begin{pmatrix}
		0 & 0\\
		0 & 0
		\end{pmatrix}\right), 2
		\right)^2
		s
		\left(\left(
		\begin{pmatrix}
		1 & 1\\
		0 & 0
		\end{pmatrix},
		\begin{pmatrix}
		0 & 0\\
		0 & 0
		\end{pmatrix}\right), 2
		\right)^2\\
		&+
		s
		\left(\left(
		\begin{pmatrix}
		1 & 1\\
		0 & 0
		\end{pmatrix},
		\begin{pmatrix}
		0 & 0\\
		1 & 1
		\end{pmatrix}\right), 2
		\right)^2
		s
		\left(\left(
		\begin{pmatrix}
		1 & 1\\
		0 & 0
		\end{pmatrix},
		\begin{pmatrix}
		0 & 0\\
		1 & 1
		\end{pmatrix}\right), 2
		\right)^2\\
		&+
		s
		\left(\left(
		\begin{pmatrix}
		1 & 1\\
		0 & 0
		\end{pmatrix},
		\begin{pmatrix}
		1 & 1\\
		0 & 0
		\end{pmatrix}\right), 1
		\right)^4
		s
		\left(\left(
		\begin{pmatrix}
		1 & 1\\
		0 & 0
		\end{pmatrix},
		\begin{pmatrix}
		1 & 1\\
		0 & 0
		\end{pmatrix}\right), 1
		\right)^4\\
		&+
		s
		\left(\left(
		\begin{pmatrix}
		1 & 1\\
		0 & 0
		\end{pmatrix},
		\begin{pmatrix}
		1 & 1\\
		1 & 1
		\end{pmatrix}\right), 2
		\right)^2
		s
		\left(\left(
		\begin{pmatrix}
		1 & 1\\
		0 & 0
		\end{pmatrix},
		\begin{pmatrix}
		1 & 1\\
		1 & 1
		\end{pmatrix}\right), 2
		\right)^2\\
		&+
		s
		\left(\left(
		\begin{pmatrix}
		1 & 1\\
		1 & 1
		\end{pmatrix},
		\begin{pmatrix}
		0 & 0\\
		0 & 0
		\end{pmatrix}\right), 2
		\right)^2
		s
		\left(\left(
		\begin{pmatrix}
		1 & 1\\
		1 & 1
		\end{pmatrix},
		\begin{pmatrix}
		0 & 0\\
		0 & 0
		\end{pmatrix}\right), 2
		\right)^2\\
		&+
		s
		\left(\left(
		\begin{pmatrix}
		1 & 1\\
		1 & 1
		\end{pmatrix},
		\begin{pmatrix}
		0 & 0\\
		1 & 1
		\end{pmatrix}\right), 2
		\right)^2
		s
		\left(\left(
		\begin{pmatrix}
		1 & 1\\
		1 & 1
		\end{pmatrix},
		\begin{pmatrix}
		0 & 0\\
		1 & 1
		\end{pmatrix}\right), 2
		\right)^2\\
		&+
		s
		\left(\left(
		\begin{pmatrix}
		1 & 1\\
		1 & 1
		\end{pmatrix},
		\begin{pmatrix}
		1 & 1\\
		0 & 0
		\end{pmatrix}\right), 2
		\right)^2
		s
		\left(\left(
		\begin{pmatrix}
		1 & 1\\
		1 & 1
		\end{pmatrix},
		\begin{pmatrix}
		1 & 1\\
		0 & 0
		\end{pmatrix}\right), 2
		\right)^2\\
		&+
		s
		\left(\left(
		\begin{pmatrix}
		1 & 1\\
		1 & 1
		\end{pmatrix},
		\begin{pmatrix}
		1 & 1\\
		1 & 1
		\end{pmatrix}\right), 1
		\right)^4
		s
		\left(\left(
		\begin{pmatrix}
		1 & 1\\
		1 & 1
		\end{pmatrix},
		\begin{pmatrix}
		1 & 1\\
		1 & 1
		\end{pmatrix}\right), 1
		\right)^4
	\end{align*}
\end{ex}

Now we give the main result of this paper, which is a generalization of Theorem~\ref{Th:CompleteJointCycle}.

\begin{thm}[Main Theorem]\label{Thm:Main}
	For $k \in \{1,\ldots,r\}$ and $j \in \{1, \ldots, \ell\}$, let $C_{kj}$ be an $\FF_q$-linear code of length $n$, where $q$ is a power of the prime number $p$. Again let $\Pi_{k}^{\ell}$ be the $\ell$-fold joint code of $C_{k1},\dots,C_{k\ell}$.
	Let $\J_{\Pi_{1}^{\ell},\ldots,\Pi_{r}^{\ell}}(x_{\bm a}:{\bm a}\in \FF_q^{\ell\times r})$ be the~$(\ell,r)$-fold complete joint weight enumerator of $\Pi_{1}^{\ell},\ldots,\Pi_{r}^{\ell}$, and 
	\[
		\CZ(\G_{\Pi_{1}^{\ell},\ldots,\Pi_{r}^{\ell}};s((g_1,\ldots,g_r),i):(g_1,\ldots,g_r)\in \G_{\Pi_{1}^{\ell},\ldots,\Pi_{r}^{\ell}},i\in\NN)
	\] 
	be the $r$-fold complete joint cycle index for $\Pi_{1}^{\ell},\ldots,\Pi_{r}^{\ell}$. 
	
	Let $T$ be a map defined as follows: 
	for each $g_1 = (\bm {c}_{11},\ldots,\bm{c}_{1n})\in \Pi_{1}^{\ell},\ldots,g_r = (\bm {c}_{r1},\ldots,\bm{c}_{rn})\in \Pi_{r}^{\ell}$, and for $i\in \{1,\ldots,n\}$,\\ 
	if 	
	$\sum_{k=1}^{r}\bm c_{ki} = {\bm 0}$, then 
	\[
		s((g_1,\ldots,g_r),1)\mapsto x_{{\bm c_{1i}\ldots \bm c_{ri}}}^{1/q^{\ell}};
	\]
	if 
	$\sum_{k=1}^{r}\bm c_{ki} \neq {\bm 0}$, then 
	\[
		s((g_1,\ldots,g_r),p)\mapsto x_{{\bm c_{1i}\ldots \bm c_{ri}}}^{p/q^{\ell}}. 
	\]
	Then we have 
	\begin{align*}
		\J_{\Pi_{1}^{\ell},\ldots,\Pi_{r}^{\ell}}(& x_{\bm a}: {\bm a}\in \FF_q^{\ell\times r})
		= \\
		T(&\CZ(\G_{\Pi_{1}^{\ell},\ldots,\Pi_{r}^{\ell}};s((g_1,\ldots,g_r),i):(g_1,\ldots,g_r)\in \G_{\Pi_{1}^{\ell},\ldots,\Pi_{r}^{\ell}},i\in\NN)). 
	\end{align*}
\end{thm}

\begin{proof}
	Let $g_k=({\bm c_{k1}},\ldots,{\bm c_{kn}})\in \Pi_{k}^{\ell}$ for $k \in \{1,\ldots,r\}$, and 
	\[
		\wt^{(\ell,r)}(g_1,\ldots,g_r)=\sharp\{i\mid \sum_{k=1}^{r}{\bm c_{ki}}\neq {\bf 0}\}. 
	\]
	If $\sum_{k=1}^{r}{\bm c_{ki}}= \bm 0$, then the $q^{\ell}$ points 
	of the form $(i,{\bm x})\in \{1,\ldots,n\}\times \FF_q^{\ell}$ 
	are all fixed by this element; 
	if $\sum_{k=1}^{r}{\bm c_{ki}}\neq {\bm 0}$, they are
	permuted in $q^{\ell}/p$ cycles of length $p$. 
	Thus, $g_k=({\bm c_{k1}},\ldots,{\bm c_{kn}})\in \Pi_{k}^{\ell}$ for $k \in \{1,\ldots,r\}$ contribute 
	\[
		s((g_1,\ldots,g_r),1)^{q^{\ell}(n-\wt^{(\ell,r)}(g_1,\ldots,g_r))}s((g_1,\ldots,g_r),p)^{(q^{\ell}/p)\wt^{(\ell,r)}(g_1,\ldots,g_r)}
	\]
	to the sum in the formula for the $r$-fold complete joint cycle index, 
	and 
	\[
		\prod_{i=1}^{n} x_{{\bm c_{1i}}\ldots{\bm c_{ri}}}
	\]
	to the sum in the formula for the $(\ell,r)$-fold complete joint weight enumerator. 
	The result follows. 
\end{proof}

\section{Average Version of Main Theorem}\label{Sec:AverageCompleteJointCycleIndex}

In~\cite{Y1989}, the notion of the average joint weight enumerators was given. Further, the notion of the average $r$-fold complete joint weight enumerators was given in~\cite{CM}. The notion of the average intersection number of codes was investigated in~\cite{Y1991}. In this section, we give the concept of the average complete joint cycle index and provide a relation with average complete joint weight enumerator of codes. We also give an analogy of the Main Theorem for the average complete joint cycle index. Finally, we also give the notion of the average intersection number for permutation groups, and establish a connection with the average intersection number of codes.

\subsection{Average of Complete Joint Cycle Index}

Let $G$ and $G^\prime$ be two permutation groups on $\Omega$, where $|\Omega|=n$. We write $G^\prime \cong G$ if $G$ and $G^\prime$ are isomorphic as permutation groups. 

\begin{df}\label{Def:AvCCyI}
	Let $G_1,\ldots,G_r$ be $r$ permutation groups on a set $\Omega$, where $|\Omega| = n$. 
	Then the \emph{$(G_1,\ldots,G_r)$-average $r$-fold complete joint cycle index} of~$G_1,\ldots,G_r$ is the polynomial
	\begin{multline*}
		\CZ_{G_1,\ldots,G_r}^{av}(s((g_1^\prime,\ldots,g_r^\prime),i))	 := \CZ^{av}(\G_{G_1^\prime,\ldots,G_r^\prime};s((g_1^\prime,\ldots,g_r^\prime),i):\\
		G_1^\prime \cong G_1,\ldots, G_r^\prime \cong G_r, (g_1^\prime,\ldots,g_r^\prime) \in \G_{G_1^\prime,\ldots,G_r^\prime}, i\in\NN),
	\end{multline*}
	in indeterminates $s((g^\prime,\ldots,g_r^\prime),i)$ where $g_1^\prime \in G_1^\prime, \ldots,g_r^\prime\in G_r^\prime$, and $i \in \NN$ defined by 
	\begin{multline*}
		\CZ_{G_1,\ldots,G_r}^{av}(s((g_1^\prime,\ldots,g_r^\prime),i))\\ := \dfrac{1}{\prod_{k=1}^{r} N_{\cong}(G_k)}
		\sum_{G_1^\prime \cong G_1} \cdots \sum_{G_r^\prime \cong G_r} \CZ_{G_1^\prime,\ldots,G_r}(s((g_{1}^\prime,\ldots,g_{r}^\prime),i)),
	\end{multline*}
	where $N_{\cong}(G_k) := \sharp\{G_k^\prime \mid G_k^\prime \cong G_k\}$.
\end{df}

In this paper we only consider the case $G_1$-average complete joint cycle index. The study of the cases $(G_1,\ldots,G_r)$-average complete joint cycle indices will be discussed in some sequel papers.
Now the \emph{$G_1$-average $r$-fold complete joint cycle index} of~$G_1,\ldots,G_r$ is the polynomial
\begin{multline*}
	\CZ_{G_1,\ldots,G_r}^{av}(s((g_1^\prime,\ldots,g_r),i))	 := \CZ^{av}(\G_{G_1^\prime,\ldots,G_r};s((g_1^\prime,\ldots,g_r),i):\\
	G_1^\prime \cong G_1, (g_1^\prime,\ldots,g_r) \in \G_{G_1^\prime,\ldots,G_r}, i\in\NN),
\end{multline*}
in indeterminates $s((g^\prime,\ldots,g_r),i)$ where $g_1^\prime \in G_1^\prime$, $g_2 \in G_2,\ldots,g_r\in G_r$, and $i \in \NN$ defined by 
\[
	\CZ_{G_1,\ldots,G_r}^{av}(s((g_1^\prime,\ldots,g_r),i)) := \dfrac{1}{N_{\cong}(G_1)}\sum_{G_1^\prime \cong G_1} \CZ_{G_1^\prime,\ldots,G_r}(s((g_{1}^\prime,\ldots,g_{r}),i)),
\]
where $N_{\cong}(G_1) := \sharp\{G_1^\prime \mid G_1^\prime \cong G_1\}$.

\begin{ex}\label{Ex:AverageCycleIndex}
	Let $S_{3}$ be the symmetric group on $\{1,2,3\}$. Again let $G_1$ and $G_2$ be two subgroup of $S_3$ such that $G_1 = \langle (1,2) \rangle$ and $G_2 = \langle (1,3,2) \rangle$. Then the subgroups of $S_{3}$ that are isomorphic as permutation group to $G_1$ are $\langle(1,2)\rangle,\langle(1,3)\rangle,\langle(2,3)\rangle$. That is $N_{\cong}(G_1) = 3$. Therefore
	\begin{align*}
		\CZ&_{G_1,G_2}^{av}(s((g_1^{\prime},g_2),i)) \\
		&= 
		\dfrac{1}{3} 
		(\CZ_{\langle(1,2)\rangle,G_2}(s((g_1^{\prime},g_2),i)) + 
		\CZ_{\langle(1,3)\rangle,G_2}(s((g_1^{\prime},g_2),i))\\ 
		&\quad\quad+\CZ_{\langle(2,3)\rangle,G_2}(s((g_1^{\prime},g_2),i)))\\
		&= 
		\dfrac{1}{3}
		(s(((1),(1)),1)^3 + s(((1),(1,2,3)),3)^1 + s(((1),(1,3,2)),3)^1 \\
		&\quad\quad+ 
		s(((1,2),(1)),1)^1 s(((1,2),(1)),2)^1 \\
		&\quad\quad+ 
		s(((1,2),(1,2,3)),1)^1 s(((1,2),(1,2,3)),2)^1 \\
		&\quad\quad+ 
		s(((1,2),(1,3,2)),1)^1 s(((1,2),(1,3,2)),2)^1\\
		&\quad\quad+
		s(((1),(1)),1)^3 + s(((1),(1,2,3)),3)^1 + s(((1),(1,3,2)),3)^1 \\
		&\quad\quad+ 
		s(((1,3),(1)),1)^1 s(((1,3),(1)),2)^1 \\
		&\quad\quad+ 
		s(((1,3),(1,2,3)),1)^1 s(((1,3),(1,2,3)),2)^1 \\
		&\quad\quad+ 
		s(((1,2),(1,3,2)),1)^1 s(((1,2),(1,3,2)),2)^1\\
		&\quad\quad+
		s(((1),(1)),1)^3 + s(((1),(1,2,3)),3)^1 + s(((1),(1,3,2)),3)^1 \\
		&\quad\quad+ 
		s(((2,3),(1)),1)^1 s(((2,3),(1)),2)^1 \\
		&\quad\quad+ 
		s(((2,3),(1,2,3)),1)^1 s(((2,3),(1,2,3)),2)^1 \\
		&\quad\quad+ 
		s(((2,3),(1,3,2)),1)^1 s(((2,3),(1,3,2)),2)^1)
	\end{align*}
\end{ex}

\begin{df}\label{Def:AverageCJWE}
We write $S_n$ for the symmetric group acting on the set $\{1,2,\dots,n\}$. 
Let $C$ be any linear code of length~$n$ over~$\FF_q$, 
and~$\bm{u} = (u_1,\ldots,u_n) \in C$. 
Then ${\sigma}(\bm{u}) := (u_{\sigma(1)},\dots, u_{\sigma(n)})$ 
for a permutation~$\sigma \in S_n$. 
Now the 
code~$C^\prime := \sigma(C):= \{{\sigma}(\bm{u}) \mid \bm{u}\in C\}$ for~$\sigma \in S_{n}$ 
is called \emph{permutationally equivalent} to $C$, 
and denoted by $C {\sim} C^\prime$.
Then the \emph{average $r$-fold complete joint weight enumerator} of codes $C_1,\ldots,C_r$ over $\FF_q$ are defined in~\cite{CM} as:
\[
	\J_{C_{1},\ldots, C_{r}}^{av}
	(x_{\bm a} : \bm a \in \FF_q^{r}) 
	:= 
	\dfrac{1}{N_{\sim}(C_{1})}
	\sum_{C_{1}^\prime \sim C_{1}} 
	\J_{C_{1}^\prime,C_{2},\ldots,C_{r}}
	(x_{\bm a} : \bm a \in \FF_q^{r}).
\]
where 
$
N_{\sim}(C_{1}) 
:= 
\sharp\{C_{1}^{\prime} \mid C_{1}^{\prime} \sim C_{1}\}.
$


We call the $G_1$-average $r$-fold complete joint cycle index
\begin{multline*}
	\CZ_{C_1,\ldots,C_r}^{av}(s((g_1^\prime,g_2,\ldots,g_r),i))
	:= \CZ^{av}(\G_{C_1^\prime,C_2,\ldots,C_r};\\
	s((g_1^\prime,g_2,\ldots,g_r),i): C_{1}^\prime \sim C_{1},
	(g_{1}^\prime,g_2,\ldots,g_r) \in \G_{C_{1}^\prime,C_2,\ldots,C_{r}}, i\in\NN)
\end{multline*}
the $G_1$-average $r$-fold complete joint cycle index for codes $C_1,\ldots,C_r$.
\end{df}

The following theorem gives a connection between the $G_1$-average of $r$-fold complete joint cycle index and the average of $r$-fold complete joint weight enumerator.

\begin{thm}\label{Th:AverageJoint}
	Let $C_1,\ldots,C_r$ be the linear codes of length $n$ over $\FF_q$, where~$q$ is a power of the prime number~$p$.
	Let $\J_{C_1,\ldots,C_r}^{av}(x_{\bm a}:{\bm a}\in \FF_q^r)$ be the average $r$-fold complete joint weight enumerator and 
	\begin{multline*}
		\CZ^{av}(\G_{C_1^\prime,C_2,\ldots,C_r};s((g_1^\prime, g_2,\ldots,g_r), i): 
		C_{1}^\prime \sim C_{1}, 
		(g_1^\prime,g_2,\ldots,g_r) 
		\in \\ \G_{C_1^\prime,C_2,\ldots, C_r}, i\in\NN)
	\end{multline*}
	be the $G_1$-average complete joint cycle index for $C_1,\ldots,C_r$. 
	
	Let $T$ be a map defined as follows: 
	for $\sigma \in S_{n}$, and $g_1 = (u_{11},\ldots,u_{1n})\in C_1, g_2 = (u_{21},\ldots,u_{2n}) \in C_2, \ldots, g_r = (u_{r1},\ldots,u_{rn}) \in C_r$, and for $i\in \{1,\ldots,n\}$, 
	if 	
	$u_{1\sigma(i)}+u_{2i}+\cdots+u_{ri} = 0$, then 
	\[
		s((g_{1}^\prime,g_2,\ldots,g_r),1)\mapsto x_{u_{1\sigma(i)} u_{2i}\ldots u_{ri}}^{1/q};
	\]
	if 
	$ u_{1\sigma(i)}+u_{2i}+\cdots+u_{ri}\neq 0$, then 
	\[
		s((g_{1}^\prime,g_2,\ldots,g_r),p)\mapsto x_{u_{1\sigma(i)} u_{2i}\ldots u_{ri}}^{p/q}. 
	\]
	Then we have 
	\begin{multline*}
		\J_{C_1,\ldots,C_r}^{av}(x_{\bm a}: {\bm a}\in \FF_q^r)
		 =
		T(\CZ^{av}(\G_{C_1^\prime,C_2,\ldots,C_r};s((g_1^\prime,g_2,\ldots,g_r),i):\\ 
		C_{1}^\prime \sim C_{1}, (g_1^\prime,g_2,\ldots,g_r) \in \G_{C_1^\prime,C_2,\ldots,C_r}, i\in\NN)). 
	\end{multline*}
\end{thm}

\begin{proof}
	Let $g_k = (u_{k1},\ldots,u_{kn})\in C_k$ for $k \in \{1,\ldots,r\}$. Then for $\sigma \in S_{n}$, $g_{k}^\prime = (u_{k\sigma(1)},\ldots,u_{k\sigma(n)})$. Again let
	\[
		\wt(g_{1}^\prime,g_{2},\ldots,g_{r})=\sharp\{i\mid u_{1\sigma(i)}+u_{2i}+\cdots+ u_{ri}\neq 0\}. 
	\]
	If $u_{1\sigma(i)}+u_{2i}+\cdots+u_{ri} = 0$, then the $q$ points of the form 
	$(i,{x})\in \{1,\ldots,n\}\times \FF_q$ 
	are all fixed by these elements; 
	if $u_{1\sigma(i)}+u_{2i}+\cdots+u_{ri}\neq 0$, they are permuted in $q/p$ cycles of length $p$, 
	Thus, $g_{k} = (u_{k1},\ldots,u_{kn})\in C_k$ for $k \in \{1,\ldots,r\}$ with $\sigma \in S_{n}$ contribute 
	\[
		s((g_{1}^\prime,g_2,\ldots,g_r),1)^{q(n-\wt(g_{1}^\prime,g_2,\ldots,g_r))}s((g_{1}^\prime,g_2,\ldots,g_r),p)^{(q/p)\wt(g_{1}^\prime,g_2,\ldots,g_r)}
	\]
	to the sum in the formula for the $G_1$-average complete joint cycle index of $C_1,\ldots,C_r$, 
	and 
	\[
		\prod_{i=1}^{n} x_{u_{1\sigma(i)}u_{2i}\cdots u_{ri}}
	\]
	to the sum in the formula for the average $r$-fold complete joint weight enumerator. 
	The result follows.
\end{proof}

\begin{df}\label{Def:rsAverageCJWE}
For 
$S_{n}^{\ell} := \underbrace{S_{n} \times \cdots \times S_{n}}_{\ell}$, 
we define the semidirect product of 
$S_{\ell}$ and $S_{n}^{\ell}$ as 
\[
	S_{\ell} \rtimes S_{n}^{\ell} := \{\iota := (\pi;\sigma_{1},\ldots,\sigma_{\ell}) \mid \pi \in S_{\ell} \text{ and } \sigma_{1},\ldots,\sigma_{\ell} \in S_{n}\}.
\]
We recall the $\ell$-fold joint code, $\Pi^{\ell}$ and for 
$\tilde{{\bm {c}}} = (\bm{c}_1,\ldots,\bm{c}_n) \in \Pi^{\ell}$,
the group $S_{\ell} \rtimes S_{n}^{\ell}$ acts on $\Pi^{\ell}$ as
\[
	\iota(\tilde{\bm {c}}):= 
	(\iota(\bm c_{1}),\ldots,\iota(\bm c_{n})):=
	\begin{pmatrix}
		a_{\pi(1)\sigma_{1}(1)} & \ldots & a_{\pi(1)\sigma_{1}(n)}\\
		a_{\pi(2)\sigma_{2}(1)} & \ldots & a_{\pi(2)\sigma_{2}(n)}\\
		\vdots&\cdots&\vdots\\
		a_{\pi(\ell)\sigma_{\ell}(1)} & \ldots & a_{\pi(\ell)\sigma_{\ell}(n)}
	\end{pmatrix},
\]
where 
$\iota(\bm c_{i}):={}^t(a_{\pi(1)\sigma_{1}(i)},\ldots,a_{\pi(\ell)\sigma_{\ell}(i)})\in \FF_q^{\ell}$. 
Then we 
call~${\Pi^{\ell}}^\prime := \iota(\Pi^{\ell}) := \{\iota(\tilde{\bm c}) \mid \tilde{\bm c} \in \Pi^{\ell}\}$ 
an \emph{equivalent $\ell$-fold joint code} to $\Pi^{\ell}$, 
and denoted by ${\Pi^{\ell}}^\prime \sim \Pi^{\ell}$. 
Now the \emph{average $(\ell,r)$-fold complete joint weight enumerator} of $\Pi_{1}^{\ell},\ldots,\Pi_{r}^{\ell}$ is defined by
\[
	\J_{\Pi_{1}^{\ell},\Pi_{2}^{\ell},\ldots,\Pi_{r}^{\ell}}^{av} 
	(x_{\bm a} : \bm a \in \FF_q^{\ell \times r})
	:= 
	\dfrac{1}{N_{\sim}({\Pi_{1}^{\ell}})}
	\sum_{{\Pi_{1}^{\ell}}^\prime \sim \Pi_{1}^{\ell}}
	\J_{{\Pi_{1}^{\ell}}^\prime,\Pi_{2}^{\ell},\ldots,\Pi_{r}^{\ell}}(x_{\bm{a}}).
\]
where
$
N_{\sim}({\Pi_{1}^{\ell}}) 
:= 
\sharp\{{\Pi_{1}^{\ell}}^{\prime} 
\mid 
{\Pi_{1}^{\ell}}^{\prime} \sim \Pi_{1}^{\ell}\}.
$
Now by Theorem~\ref{Th:GenMacWilliams}, we have the generalized MacWilliams identity for the average $(\ell,r)$-fold complete joint weight enumerator as follows:
\begin{multline*}
	\J_{\widehat{\Pi}_{1}^{\ell},\ldots,\widehat{\Pi}_{r}^{\ell}}^{av}(x_{\bm{a}} : \bm{a} \in \FF_q^{\ell \times r}) = \dfrac{1}{|\Pi_{1}^{\ell}|^{\delta(\Pi_{1}^{\ell},\widehat{\Pi}_{1}^{\ell})}\cdots|\Pi_{r}^{\ell}|^{\delta(\Pi_{r}^{\ell},\widehat{\Pi}_{r}^{\ell})}}\\ 
	T^{\delta(\Pi_{1}^{\ell},\widehat{\Pi}_{1}^{\ell})} \otimes \cdots \otimes T^{\delta(\Pi_{r}^{\ell},\widehat{\Pi}_{r}^{\ell})} \J_{\Pi_{1}^{\ell},\ldots,\Pi_{r}^{\ell}}^{av}(x_{\bm{a}}),
\end{multline*}
{where $T^{0}$ represents the identity matrix~$I$}.

We call the $G_1$-average $r$-fold complete joint cycle index
\begin{multline*}
	\CZ_{\Pi_{1}^{\ell},\ldots,\Pi_{r}^{\ell}}^{av}(s((g_1^\prime,g_2,\ldots,g_r),i))
	:= 
	\CZ^{av}(\G_{{\Pi_{1}^{\ell}}^\prime,\Pi_{2}^{\ell},\ldots,\Pi_{r}^{\ell}};\\
	s((g_1^\prime,g_2,\ldots,g_r),i):
	{\Pi_{1}^{\ell}}^\prime \sim \Pi_{1}^{\ell},
	(g_{1}^\prime,g_2,\ldots,g_r) \in \G_{{\Pi_{1}^{\ell}}^\prime,\Pi_{2}^{\ell},\ldots,\Pi_{r}^{\ell}}, i\in\NN)
\end{multline*}
the $G_1$-average $r$-fold complete joint cycle index for $\ell$-fold joint codes $\Pi_{1}^{\ell},\ldots,\Pi_{r}^{\ell}$.
\end{df}

In the following theorem, we give a relationship between the average $(\ell,r)$-fold complete joint weigh enumerator and the $G_1$-average $r$-fold complete joint cycle index for $\ell$-fold joint codes as a generalization of Theorem~\ref{Th:AverageJoint}.

\begin{thm}\label{Thm:rsAverage}
	For $k \in \{1,\ldots,r\}$ and $j \in \{1, \ldots, \ell\}$, let $C_{kj}$ be an $\FF_q$-linear code of length $n$, where $q$ is a power of the number $p$. Again let $\Pi_{k}^{\ell}$ be an $\ell$-fold joint code of $C_{k1},\dots,C_{k\ell}$.
	Let $\J_{\Pi_{1}^{\ell},\ldots,\Pi_{r}^{\ell}}^{av}(x_{\bm a}:{\bm a}\in \FF_q^{\ell\times r})$ be the average $(\ell,r)$-fold complete joint weight enumerator of $\Pi_{1}^{\ell},\ldots,\Pi_{r}^{\ell}$, and 
	\begin{multline*}
		\CZ^{av}(\G_{\iota\Pi_{1}^{\ell},\ldots,\Pi_{r}^{\ell}};
		s((g_1^{\prime},\ldots,g_r),i):
		{\Pi_{1}^{\ell}}^\prime \sim \Pi_{1}^{\ell},
		(g_1^{\prime},\ldots,g_r)\in \G_{{\Pi_{1}^{\ell}}^\prime,\ldots,\Pi_{r}^{\ell}}, \\i\in\NN)
	\end{multline*} 
	be the $G_1$-average $r$-fold complete joint cycle index for $\Pi_{1}^{\ell},\ldots,\Pi_{r}^{\ell}$. 
	
	Let $T$ be a map defined as follows: 
	for $\iota = (\pi;\sigma_{1},\ldots,\sigma_{\ell}) \in S_{\ell}\rtimes S_{n}^{\ell}$, and $g_1 = (\bm {c}_{11},\ldots,\bm{c}_{1n})\in \Pi_{1}^{\ell},\ldots,g_r = (\bm {c}_{r1},\ldots,\bm{c}_{rn})\in \Pi_{r}^{\ell}$, and for $i\in \{1,\ldots,n\}$,
	if 	
	$\iota(\bm c_{1i}) + \bm c_{2i} + \cdots + \bm c_{ri} = {\bm 0}$, then 
	\[
		s((g_1^\prime,\ldots,g_r),1)
		\mapsto 
		x_{{\iota(\bm c_{1i})\bm c_{2i} \ldots \bm c_{ri}}}^{1/q^{\ell}};
	\]
	if 
	$\iota(\bm c_{1i}) + \bm c_{2i} + \cdots + \bm c_{ri} \neq {\bm 0}$, then 
	\[
		s((g_1^\prime,\ldots,g_r),p)\mapsto x_{{\iota(\bm c_{1i})\bm c_{2i}\ldots \bm c_{ri}}}^{p/q^{\ell}}. 
	\]
	Then we have 
	\begin{multline*}
		\J_{\Pi_{1}^{\ell},\ldots,\Pi_{r}^{\ell}}^{av}
		(x_{\bm a}: {\bm a}\in \FF_q^{\ell\times r})=
		T(\CZ^{av}(\G_{{\Pi_{1}^{\ell}}^\prime,\ldots,\Pi_{r}^{\ell}};
		s((g_1^{\prime},\ldots,g_r),i):\\
		{\Pi_{1}^{\ell}}^\prime \sim \Pi_{1}^{\ell},
		(g_1^{\prime},\ldots,g_r)\in \G_{{\Pi_{1}^{\ell}}^\prime,\ldots,\Pi_{r}^{\ell}}, i\in\NN)). 
	\end{multline*}
\end{thm}

\begin{proof}
	Let $g_k=({\bm c_{k1}},\ldots,{\bm c_{kn}})\in \Pi_{k}^{\ell}$ for $k \in \{1,\ldots,r\}$. Then for 
	$\iota \in S_{\ell} \rtimes S_{n}^{\ell}$, 
	$g_k^\prime = (\iota({\bm c_{k1}}),\ldots,\iota(\bm c_{kn}))$. 
	Again let 
	\[
		\wt^{(\ell,r)}(g_1^\prime,\ldots,g_r)=
		\sharp\{i\mid \iota(\bm c_{1i}) + \bm c_{2i} + \cdots + \bm c_{ri}\neq {\bf 0}\}. 
	\]
	If $\iota(\bm c_{1i}) + \bm c_{2i} + \cdots + \bm c_{ri} = \bm 0$, then the $q^{\ell}$ points 
	of the form $(i,{\bm x})\in \{1,\ldots,n\}\times \FF_q^{\ell}$ 
	are all fixed by this element; 
	if $\iota(\bm c_{1i}) + \bm c_{2i} + \cdots + \bm c_{ri} \neq {\bm 0}$, they are
	permuted in $q^{\ell}/p$ cycles of length $p$. 
	Thus, $g_k=({\bm c_{k1}},\ldots,{\bm c_{kn}})\in \Pi_{k}^{\ell}$ for $k \in \{1,\ldots,r\}$ with $\iota\in S_{\ell} \rtimes S_{n}^{\ell}$, contribute 
	\[
		s((g_1^\prime,\ldots,g_r),1)^{q^{\ell}(n-\wt^{(\ell,r)}(g_1^\prime,\ldots,g_r))}s((g_1^\prime,\ldots,g_r),p)^{(q^{\ell}/p)\wt^{(\ell,r)}(g_1^\prime,\ldots,g_r)}
	\]
	to the sum in the formula for the $G_1$-average complete joint cycle index of $\Pi_{1}^{\ell},\ldots,\Pi_{r}^{\ell}$, 
	and 
	\[
		\prod_{i=1}^{n} x_{\iota(\bm c_{1i})\bm c_{2i}\cdots{\bm c_{ri}}}
	\]
	to the sum in the formula for the average $(\ell,r)$-fold complete joint weight enumerator. 
	The result follows. 
\end{proof}

\subsection{Average Intersection Number}

Let $C$ and $D$ be two linear codes of length~$n$ over $\FF_q$. The notion of the average intersection number of $C$ and $D$ was defined in~\cite{Y1991} as follows:
\[
	\Delta(C,D) 
	:= 
	\dfrac{1}{N_{\sim}(C)} 
	\sum_{C^\prime \sim C} 
	|C^\prime \cap D|.
\] 
where 
$
N_{\sim}(C) 
:= 
\sharp\{C^{\prime} \mid C^{\prime} \sim C\}.
$
We now give the concept of the average intersection number of two permutation groups.

\begin{df}\label{Def:GrAvIN}
	Let $G$ and $H$ be two permutation groups on a set $\Omega$, where $|\Omega| = n$. The \emph{average intersection number} of~$G$ and $H$ is defined as
	\[
		\mathfrak{I}^{av}(G,H) := 
		\dfrac{1}{N_{\cong}(G)}\sum_{G^\prime \cong G}
		|G^\prime \cap H|.
	\]
\end{df}

\begin{ex}\label{Ex:IntersectionNumber}
	From Example~\ref{Ex:AverageCycleIndex}, we have $\mathfrak{I}^{av}(G_1,G_2) = 1$.
\end{ex}

Now we recall Definition~\ref{Def:CompleteJointCycleIndex}, and construct from $C$ and $D$ two permutation groups $G(C)$ and $H(D)$. 
Then for 
$C^\prime \sim C$,
we have $G^\prime(C^\prime) \cong G(C)$. We call the average intersection number~$\mathfrak{I}^{av}(G(C),H(D))$ the average intersection number for codes $C$ and $D$.

We enclose the section with the following result. Since the proof of the following theorem is straightforward, we omit the detail.

\begin{thm}\label{Th:IntersectionNumber}
	Let $C$ and $D$ be two codes of length~$n$ over $\FF_q$. Again let $G(C)$ and $H(D)$ be two permutation groups constructed from $C$ and $D$ respectively. Then $\mathfrak{I}^{av}(G(C),H(D)) = \Delta(C,D)$.
\end{thm}

\section{$\ZZ_{k}$-code Analogue of Main Result}\label{Sec:ZkAnalogue}

In~\cite{BDHO}, the authors introduced the concept of $\ZZ_{k}$-linear codes. In this section, we give a $\ZZ_{k}$-linear code analogue of Theorem~\ref{Thm:Main}.

Let $\ZZ_{k}$ be the ring 
of integers modulo $k$, where $k$ 
is a positive integer. 
In this paper, we always assume that $k\geq 2$ and 
we take the set $\ZZ_{k}$ to be 
$\{0,1,\ldots,k-1\}$.  
A $\ZZ_{k}$-linear code $C$ of length $n$
is a $\ZZ_{k}$-submodule of $\ZZ_{k}^n$.
The \emph{inner product} of two elements 
$\bm{u},\bm{v}\in \ZZ_{k}^n$ is defined as:
\[
	\bm{u} \cdot \bm{v} := u_1v_1 + u_2v_2 + \cdots + u_nv_n \pmod k,
\]
where $\bm{u} = (u_1,u_2,\ldots, u_n)$ and $\bm{v} = (v_1,v_2,\ldots,v_n)$.

We denote, $\Pi^{\ell} := C_{1} \times\cdots\times C_{\ell}$,
where $C_{1},\ldots,C_{\ell}$ are the $\ZZ_{k}$-linear codes of length $n$. We call $\Pi^{\ell}$ as \emph{$\ell$-fold joint code} of $C_{1},\ldots,C_{\ell}$. We denote an element of $\Pi^{\ell}$ by 
\[
	{\tilde{\bm {c}}} := 
	({\bm c_{1}},\ldots,{\bm c_{n}}):=
	\begin{pmatrix}
		a_{11}&\ldots&a_{1n}\\
		a_{21}&\ldots&a_{2n}\\
		\vdots&\cdots&\vdots\\
		a_{\ell 1}&\ldots&a_{\ell n}
	\end{pmatrix},
\]
where 
${\bm c_{i}}:={}^t(a_{1i},\ldots,a_{\ell i})\in \ZZ_{k}^{\ell}$
and 
$\mu_{j}(\tilde{\bm{c}}):= (a_{j1},\ldots, a_{jn}) \in C_{j}$.

Now let $\Pi_{1}^{\ell},\ldots,\Pi_{r}^{\ell}$ 
be the $\ell$-fold joint codes (not necessarily the same) over $\ZZ_{k}$. 
For 
$\nu \in \{1,\ldots,r\}$, 
we denote, $\Pi_{\nu}^{\ell} := C_{\nu 1} \times \cdots \times C_{\nu \ell}$, 
where $C_{\nu 1},\ldots,C_{\nu \ell}$ 
be the linear codes of length~$n$ over $\ZZ_{k}$.
An element of $\Pi_{\nu}^{\ell}$ is denoted by
\[
	\tilde{\bm {c}}_{\nu}:= 
	({\bm c_{\nu 1}},\ldots,{\bm c_{\nu n}}):=
	\begin{pmatrix}
		a_{11}^{(\nu)}&\ldots&a_{1n}^{(\nu)}\\
		a_{21}^{(\nu)}&\ldots&a_{2n}^{(\nu)}\\
		\vdots&\cdots&\vdots\\
		a_{\ell 1}^{(\nu)}&\ldots&a_{\ell n}^{(\nu)}
	\end{pmatrix},
\]
where 
${\bm c_{\nu i}}:={}^t(a_{1i}^{(\nu)},\ldots,a_{\ell i}^{(\nu)})\in \ZZ_{k}^{\ell}$. 
and 
$\mu_{j}(\tilde{\bm{c}}_{\nu}):= (a_{j1}^{(\nu)},\ldots, a_{jn}^{(\nu)}) \in C_{\nu j}$.
Then the \emph{$(\ell,r)$-fold complete joint weight enumerator} of $\Pi_{1}^{\ell},\ldots,\Pi_{r}^{\ell}$ is defined as follows:
\[
	\J_{\Pi_{1}^{\ell},\ldots,\Pi_{r}^{\ell}}(x_{\bm a} : \bm a \in \ZZ_{k}^{\ell \times r}) := 
	\sum_{\tilde{\bm {c}}_1\in \Pi_{1}^{\ell},\ldots,\tilde{\bm {c}}_r \in \Pi_{r}^{\ell}}
	\prod x_{\bm {a}}^{n_{\bm a}(\tilde{\bm {c}}_1,\ldots, \tilde{\bm {c}}_r)},
\]
where ${n_{\bm a}(\tilde{\bm {c}}_1,\ldots, \tilde{\bm {c}}_r)}$ 
denotes the number of $i$ such that 
${\bm a}=(\bm c_{1i},\ldots,\bm c_{ri})$. 

We have the MacWilliams identity for the complete weight enumerator of a $\ZZ_{k}$-linear code $C$ as follows. 

\begin{thm}[\cite{BDHO}]\label{Th:ZkMacWilliams}
	Let $C$ be a linear code of length~$n$ over $\ZZ_{k}$, 
	and $\mathcal{C}_{C}(x_a : a \in \ZZ_{k})$ be the complete weight enumerator of $C$. Again let $\eta_{k}$ be the $k$th primitive root of unity. Then we have
	\[
		\mathcal{C}_{C^\perp}(x_a : a \in \ZZ_{k}) = 
		\dfrac{1}{|C|} T \cdot \mathcal{C}_C(x_a),
	\]
	where $T = \left(\eta_{k}^{ab}\right)_{a,b \in \ZZ_{k}}$.
\end{thm}

We adopt the following notation for simplicity,
\[
	{\Pi^{\ell}}^\perp := C_1^\perp \times \cdots \times C_{\ell}^\perp 
	\text{ and } 
	|{\Pi^{\ell}}| := |C_1|\times \cdots \times |C_{\ell}|.
\] 

Now we give the $\ZZ_{k}$-linear code version of Theorem~\ref{Th:GenMacWilliams}. The proof is similar to that of Theorem~\ref{Th:GenMacWilliams}. Therefore, we omit to give the detail.

\begin{thm}\label{Th:ZkGenMacWilliams}
	The MacWilliams identity for the $(\ell,r)$-fold complete joint weight enumerator of $\ell$-fold joint codes $\Pi_{1}^{\ell},\ldots,\Pi_{r}^{\ell}$ over $\ZZ_{k}$ is given by
	\begin{multline*}
	\J_{\widehat{\Pi}_{1}^{\ell},\ldots,\widehat{\Pi}_{r}^{\ell}}
	(x_{\bm a} : \bm a \in \ZZ_{k}^{\ell \times r}) 
	= 
	\dfrac{1}{|\Pi_{1}^{\ell}|^{\delta(\Pi_{1}^{\ell},\widehat{\Pi}_{1}^{\ell})} \cdots|\Pi_{r}^{\ell}|^{\delta(\Pi_{r}^{\ell},\widehat{\Pi}_{r}^{\ell})}}\\ 
	\left(T^{\delta(\Pi_{1}^{\ell},\widehat{\Pi}_{1}^{\ell})}\right)^{\otimes\ell} 
	\otimes \cdots \otimes 
	\left(T^{\delta(\Pi_{r}^{\ell},\widehat{\Pi}_{r}^{\ell})}\right)^{\otimes\ell}  
	\J_{\Pi_{1}^{\ell},\ldots,\Pi_{r}^{\ell}}(x_{\bm a}),
	\end{multline*}
	where
	\[
		\delta(\Pi^{\ell},\widehat{\Pi}^{\ell}) := 
		\begin{cases}
			0 & \mbox{if} \quad \widehat{\Pi}^{\ell} = \Pi^{\ell}, \\
			1 & \mbox{if} \quad \widehat{\Pi}^{\ell} = {\Pi^{\ell}}^\perp . 
		\end{cases}
	\]
\end{thm}

Now we construct from $\Pi^{\ell}$ a permutation group $G(\Pi^{\ell})$. The group we construct is the additive group of $\Pi^{\ell}$. We let it act on the set $\{1,\ldots,n\}\times \ZZ_{k}^{\ell}$ in the following way:
$({\bm c_{1}},\ldots,{\bm c_{n}})$
acts as the permutation
\[
	\left(i,
	\begin{pmatrix}
		x_{1}\\
		x_{2}\\
		\vdots\\
		x_{\ell}
	\end{pmatrix}
	\right)
	\mapsto 
	\left(i,
	\begin{pmatrix}
		x_{1}+a_{1i}\\
		x_{2}+a_{2i}\\
		\vdots\\
		x_{\ell}+a_{\ell i}
	\end{pmatrix}
	\right) 
\]
of the set $\{1,\ldots,n\}\times \ZZ_{k}^{\ell}$.
Now let $G_{1}(\Pi_{1}^{\ell}),\ldots,G_{r}(\Pi_{r}^{\ell})$ be $r$ permutation groups. 
We define the \emph{product} of $r$ permutations $({\bm c_{11}}, \ldots, {\bm c_{1n}}) \in \Pi_{1}^{\ell}, \ldots, ({\bm c_{r1}}, \ldots, {\bm c_{rn}}) \in \Pi_{r}^{\ell}$ as follows:
\[
	\left(i,
	\begin{pmatrix}
		x_{1}\\
		x_{2}\\
		\vdots\\
		x_{\ell}
	\end{pmatrix}
	\right)
	\mapsto 
	\left(i,
	\begin{pmatrix}
		x_{1}+\beta_{1}^{i} \pmod k\\
		x_{2}+\beta_{2}^{i} \pmod k\\
		\vdots\\
		x_{\ell}+\beta_{\ell}^{i} \pmod k
	\end{pmatrix}
	\right) 
\]
of a set $\{1, \dots, n\} \times \ZZ_{k}^{\ell}$, 
where 
$\beta_{j}^{i} := a_{ji}^{(1)} + \cdots + a_{ji}^{(r)} \pmod k$. 

The following theorem is a $\ZZ_k$-linear code analogue of Theorem~\ref{Thm:Main}.

\begin{thm}\label{Th:ZkAnalogue}
	For $\nu \in \{1,\ldots,r\}$ and $j \in \{1, \ldots, \ell\}$, let $C_{\nu j}$ be 
	the~$\ZZ_{k}$-linear code of length $n$. 
	Again $\Pi_{\nu}^{\ell}$ be the $\ell$-fold joint code of $C_{\nu 1},\dots,C_{\nu \ell}$.
	Let $\J_{\Pi_{1}^{\ell},\ldots,\Pi_{r}^{\ell}}(x_{\bm a}:{\bm a}\in \ZZ_{k}^{\ell\times r})$ be the~$(\ell,r)$-fold complete joint weight enumerator of $\Pi_{1}^{\ell},\ldots,\Pi_{r}^{\ell}$, and 
	$\CZ(\G_{\Pi_{1}^{\ell},\ldots,\Pi_{r}^{\ell}};s((g_1,\ldots,g_r),i):(g_1,\ldots,g_r)\in \G_{\Pi_{1}^{\ell},\ldots,\Pi_{r}^{\ell}},i\in\NN)$ be the $r$-fold complete joint cycle index for $\Pi_{1}^{\ell},\ldots,\Pi_{r}^{\ell}$. 
	
	Let $T$ be a map defined as follows: 
	for each $g_1 = (\bm {c}_{11},\ldots,\bm{c}_{1n})\in \Pi_{1}^{\ell},\ldots,g_r = (\bm {c}_{r1},\ldots,\bm{c}_{rn})\in \Pi_{r}^{\ell}$, and for $i\in \{1,\ldots,n\}$,\\ 
	if 	
	$\sum_{\nu=1}^{r}\bm c_{\nu i} \equiv {\bm 0} \pmod k$, then 
	\[
		s((g_1,\ldots,g_r),1)\mapsto x_{{\bm c_{1i}\ldots \bm c_{ri}}}^{1/k^{\ell}};
	\]
	if 
	$\sum_{\nu=1}^{r}\bm c_{\nu i} \not\equiv {\bm 0} \pmod k$, then
	\[
		s((g_1,\ldots,g_r),k/\gcd(\beta_{1}^{i},\ldots,\beta_{\ell}^{i},k))
		\mapsto 
		x_{{\bm c_{1i}}\ldots{\bm c_{ri}}}^{(k/\gcd(\beta_{1}^{i},\ldots,\beta_{\ell}^{i},k))/k^{\ell}}. 
	\]
	Then we have 
	\begin{align*}
		\J_{\Pi_{1}^{\ell},\ldots,\Pi_{r}^{\ell}}(& x_{\bm a}: {\bm a}\in \ZZ_k^{\ell\times r})
		= \\
		T(&\CZ(\G_{\Pi_{1}^{\ell},\ldots,\Pi_{r}^{\ell}};s((g_1,\ldots,g_r),i):(g_1,\ldots,g_r)\in \G_{\Pi_{1}^{\ell},\ldots,\Pi_{r}^{\ell}},i\in\NN)). 
	\end{align*}
\end{thm}

\begin{proof}
	Let $g_{\nu}=({\bm c_{\nu 1}},\ldots,{\bm c_{\nu n}})\in \Pi_{\nu}^{\ell}$ for $\nu \in \{1,\ldots,r\}$, and 
	\[
		\wt^{(\ell,r)}(g_1,\ldots,g_r)=\sharp\{i\mid \sum_{\nu=1}^{r}{\bm c_{\nu i}}\not\equiv {\bf 0} \pmod k\}. 
	\]
	If $\sum_{\nu=1}^{r}{\bm c_{\nu i}} \equiv \bm 0 \pmod k$, then the $k^{\ell}$ points 
	of the form $(i,{\bm x})\in \{1,\ldots,n\}\times \ZZ_{k}^{\ell}$ 
	are all fixed by this element; 
	if $\sum_{\nu=1}^{r}{\bm c_{\nu i}}\not\equiv {\bm 0} \pmod k$, they are
	permuted in 
	\[
		k^{\ell}/(k/\gcd(\beta_{1}^{i},\ldots,\beta_{\ell}^{i},k)) \text{ cycles of length } 
		k/\gcd(\beta_{1}^{i},\ldots,\beta_{\ell}^{i},k).
	\] 
	Then the result follows from the argument in the proof of Theorem~\ref{Thm:Main}. 
\end{proof}

\section*{Acknowledgements}

\noindent
The authors would like to thank the anonymous
reviewers for their beneficial comments on an earlier version of the manuscript.
The second named author is supported by JSPS KAKENHI (18K03217). 




\end{document}